\documentclass[english,12pt]{elsarticle}
\usepackage[english]{babel}
\usepackage{geometry}
\usepackage{hyperref}
\usepackage{amsthm}
\usepackage{amsmath}
\usepackage{amssymb}
\usepackage{graphicx}
\usepackage{thmtools, thm-restate}
\makeatletter
\theoremstyle{plain}
\newtheorem{theorem}{Theorem}[section]
\newtheorem{definition}{Definition}[section]
\newtheorem{lemma}{Lemma}[section]
\newtheorem{remark}{Remark}[section]
\newtheorem{proposition}{Proposition}[section]
\newtheorem{corollary}{Corollary}[section]
\newtheorem{notation}{Notation}[section]

\usepackage{babel}
\usepackage{tikz}
\usepackage{tikz-cd}
\numberwithin{equation}{section}
\usepackage{caption}
\usepackage{comment}

\usepackage{etoolbox}
\apptocmd{\sloppy}{\hbadness 10000\relax}{}{}

\newenvironment{psmallmatrix}
{\left(\begin{smallmatrix}}
	{\end{smallmatrix}\right)}

\let\today\relax
\makeatletter
\def\ps@pprintTitle{%
	\let\@oddhead\@empty
	\let\@evenhead\@empty
	\def\@oddfoot{\footnotesize\itshape
		 \hfill\today}%
	\let\@evenfoot\@oddfoot
}
\makeatother

\begin{document}
	
\begin{frontmatter}
		
	\title{Classification of Hopf superalgebras associated with quantum special linear superalgebra at roots of unity using Weyl groupoid}
	
	\author[1]{Alexander Mazurenko \footnote{Corresponding Author: Alexander Mazurenko, mazurencoal@gmail.com}}
	\ead{mazurencoal@gmail.com}
	
	\author[2]{Vladimir A. Stukopin}
	\ead{stukopin@mail.ru}
	
	\address[1]{MCCME (Moscow Center for Continuous Mathematical Education)}
	\address[2]{MIPT (Moscow Institute of Physics and Technology) \\ SMI of VSC RAS (South Mathematical Institute of Vladikavkaz Scientific Center of Russian Academy of Sciences) \\ MCCME (Moscow Center for Continuous Mathematical Education)}
	
	\begin{abstract}
				
		We summarize the definition of the Weyl groupoid using supercategory approach in order to investigate quantum superalgebras at roots of unity.  We show how the structure of a Hopf superalgebra on a quantum superalgebra is determined by the quantum Weyl groupoid.  The Weyl groupoid of $\mathfrak{sl}(m|n)$ is constructed to this end as some supercategory. We prove that in this case quantum superalgebras associated with Dynkin diagrams are isomorphic as superalgebras. It is shown how these quantum superalgebras considered as Hopf superalgebras are connected via twists and isomorphisms. We explicitly construct these twists  using the Lusztig isomorphisms considered as elements of the Weyl quantum groupoid.  We build a PBW basis for each quantum superalgebra, and investigate how quantum superalgebras are connected with their classical limits, i. e. Lie superbialgebras. We find explicit multiplicative formulas for universal $R$-matrices, describe relations between them for each realization and classify Hopf superalgebras and triangular structures for the quantum superalgebra $U_q(\mathfrak{sl}(m|n))$.
		
	\end{abstract}
	
	\begin{keyword}
	 Dynkin diagram \sep Lie superalgebras \sep Lie superbialgebras \sep Lusztig isomorphisms \sep PBW basis \sep quantum superalgebras \sep universal R-matrix \sep quantum Weyl groupoid \sep supercategory
	
	 MSC-class: 16W35 (Primary) 16W55, 17B37, 81R50, 16W30 (Secondary)
	\end{keyword}
	
\end{frontmatter}

\section{Introduction}
\label{Int}

In this paper we investigate quantum deformation \cite{D87} of the Lie superalgebra at roots of unity in the case of quantum superalgebra $sl(m|n)$, where $m \ne n$ and $m, n > 0$. Our considerations are based on a Weyl groupoid, which we define as some supercategory, see Definition \ref{def:WG}. We show how to associate quantum superalgebras at roots of unity to Dynkin diagrams in Section \ref{subs:dqsru}. Although the definition of the Weyl groupoid used by us is a slight modification of the definition given in  \cite{HY08}, it seems to us to be more suitable for solving further problems of geometrization and categorification of the main objects of the theory of quantum superalgebras. One of our main results is Theorem \ref{th:ltisom} where we show that the two realizations are isomorphic as superalgebras. We investigate how to build a PBW basis for each realization in Theorem \ref{th:pbwbasis}. In Theorem \ref{th:linkHosalgstrd} we show how the two realizations are connected as Hopf superalgebras. We also compute universal $R$-matrices and describe relations between them for each realization using the formulas for the twists corresponding to reflections in the Weyl groupoid.

Our work is motivated by results obtained in \cite{HY08} and reformulated in \cite{CH09}. In these papers is defined a Weyl groupoid. In \cite{S11} is investigated a Weyl groupoid related to the Lie superalgebras. The case of quantum superalgebras is considered in \cite{HSTY08}. We were inspired also by results obtained in \cite{KT94}, \cite{LS91} and \cite{LSS93}. We also note that this paper generalizes and refines results obtained in our previous article \cite{MS20}.

We study the structures of Hopf superalgebras and triangular structures defined by universal R-matrices and describe the connection between them. This connection is described by twists in terms of the elements of the Weyl quantum groupoid. Thus, the Weyl groupoid plays a central and connecting role in our work. We define the Weyl groupoid in a form convenient for the purpose of our work. Nonetheless, our definition is essentially equivalent to the generally accepted one.

It should be emphasized that the central role in our reasoning is played by the construction of Weyl quantum groupoid.  Following V. Serganova, H. Yamane, I. Heckenberger  ideas, we define the Weyl groupoid as a supercategory in generality sufficient for our purposes of studying quantum superalgebras. Based on this abstract definition, we give an explicit realization of the Weyl quantum groupoid in terms of isomorphisms of quantum superalgebras generated by isomorphisms, which are induced by reflections (relatively both  even and odd roots).  We interpret the elements of the quantum  Weyl  groupoid as adjoint maps given by the elements of the quantum superalgebra $U_q(sl(m|n)$ (or some its extension). Thus, we associate the elements of the quantum Weyl groupoid with the group-like elements of the quantum superalgebra.  Among other things, we note that these elements satisfy some analogue of the braid group relations and define braid groupoid structure. It should be noted that the classical Weyl groupoid, acting by reflections with respect to odd roots, permutes non-conjugate Borel subalgebras. Thus, it permutes the corresponding non-isomorphic Manin triples defined by these Borel subalgebras. This gives a natural description and, moreover, a classification of the structures of Lie  bisuperalgebras on superalgebra $sl(m|n)$. We confine ourselves here to a detailed consideration of the case $sl(2|1)$. Similarly, the action of the quantum Weyl groupoid allows one to classify the structures of the Hopf superalgebra on the quantized universal enveloping superalgebra $U_q(sl(m|n))$. We describe a connection between the structures of the Hopf superalgebras and give explicit formulas for both twists which intertwines different comultiplication operations and for universal $R$-matrices of various Hopf superalgebra structures on $U_q(sl(m|n))$. We consider in detail the case $U_q(sl(2|1))$ We explicitly describe the connection of these constructions with the Weyl quantum groupoid.

We consider the above constructions by specializing the quantization parameter $q$ with a value equal to the odd root of unity $q^n=1, n=2m+1$.

To make the basic constructions as explicit as possible, we restrict ourselves to considering a particular case and investigate only the Weyl groupoid of the Lie superalgebra $sl(m|n)$, where $m \ne n$ and $m, n > 0$. Nonetheless, all our considerations can be adopt to the more general case of an arbitrary basic Lie superalgebra. Thus, our definition is based on the definition of the Weyl groupoid given in \cite{HY08}, \cite{CH09}, \cite{S11} and contains the classical and quantum versions of the Weyl groupoid. We also give an explicit construction of the Weyl quantum groupoid using Lustig isomorphisms in the spirit of the \cite{LS91}, \cite{LSS93} and \cite{L10} (see also \cite{H10}, \cite{AA17}). Using this explicit description of the Weyl quantum  groupoid  we investigate Hopf superalgebras  structures and triangular structures associated with Dynkin diagrams and show how they are connected via twists and isomorphisms.

Note that our initial goal, partially motivated by applications, was to study the Hopf structures of quantum affine superalgebras. The second goal  is to try to partially clarify the categorical and geometric nature of quantum superalgebras and to show how these structures are determined by simple combinatorial data. But we decided to introduce these important and general constructions based on the Weyl quantum groupoid in the simplest possible form, so that technical difficulties would not impede the understanding of simple basic ideas.

It is worth mentioning that we use the traditional language of the theory of quantum groups (which are deformations of Lie bialgebras), but our constructions are equivalent to the constructions used by algebraists in the abstract theory of Hopf and Nichols algebras (see \cite{AA17}, \cite{AngY15}, \cite{AY11}, \cite{H10}).

We will now give an outline of this paper. In Section \ref{subs:dsl} we recall basic facts about Lie superalgebras, and remind some categorical definitions about supercategories. Next we describe Lie superalgebra $sl(2|1)$ and show how to endow it with the Lie superbialgebra structure.

Section \ref{sec:WG} is divided in three parts. In Subsection \ref{WGD} we give the definition of Cartan scheme, use it to construct a category called Weyl groupoid and show how to build $\mathcal{W}(\mathcal{R})$ the Weyl groupoid of the Lie superalgebra $sl(2|1)$ in Subsection \ref{sc:Wgsl}. Next in Subsection \ref{sc:CcLs} we construct a faithful covariant functor from the $\mathcal{W}(\mathcal{R})$ to the category of Lie superalgebras associated with Dynkin diagrams. Moreover, we show how to endow these Lie superalgebras with the structure of Lie superbialgebras and investigate how they are related to each other.

Section \ref{sec:Wgqsru} is divided in six parts. In Subsection \ref{subs:ques} we recall the definition of the quantized universal enveloping superalgebras. Next in Subsection \ref{subs:dqsru} it is shown how to associate with Dynkin diagram quantum superalgebra at roots of unity. Subsection \ref{sec:chst} contains auxiliary categorical definitions and results about Hopf superalgebras. In Subsection \ref{subs:LTI} we construct a faithful covariant functor from the $\mathcal{W}(\mathcal{R})$ to the category of superalgebras associated with Dynkin diagrams and prove that these superalgebras are isomorphic. In Subsection \ref{subs:pbwb} we show how to build a PBW basis for these superalgebras in cas of $sl(2|1)$ (see, also \cite{HY10}). In Subsection \ref{subs:hssurm} we investigate braided Hopf superalgebras associated with Dynkin diagrams and show how they are connected via twists and isomorphisms. The central result is Theorem \ref{th:sufcondaut} which allows us to classify all Hopf superalgebras which we investigate.

In this paper we use the following notation. Let $\mathbb{N}$, $\mathbb{Z}$ and $\mathbb{Q}$ denote the sets of natural numbers, integers and rational numbers,  respectively. Let $\Bbbk$ be an algebraically closed field of characteristic zero. We also use Iverson bracket defined by $ [P] = \begin{cases} 1 \text{ if } P \text{ is true;} \\ 0 \text{ otherwise}, \end{cases}$ where $P$ is a statement that can be true or false.

\section{Special Lie superalgebra $sl(2|1)$}
\label{subs:dsl}

As for the terminology concerning Lie superalgebras, we refer to \cite{K77}, \cite{FSS89}.

A super vector space (superspace) $V$ over field $\Bbbk$ is a $\Bbbk$-vector space endowed with a $\mathbb{Z}_2$-grading, in other words, it writes as a direct sum of two vector spaces $V = V_{\bar{0}} \oplus V_{\bar{1}}$ such as $V_{\bar{0}}$ is the even part and $V_{\bar{1}}$ is the odd part. Define a parity function $\lvert \cdot \rvert: V \to \mathbb{Z}_2$ for a homogeneous element $x$ in a superspace by $|x| = \bar{a}$, where $v \in V_{\bar{a}}$ and $\bar{a} \in \mathbb{Z}_2$. A superalgebra $A$ over the field $\Bbbk$ is a $\mathbb{Z}_2$-graded algebra $A = A_{\bar{0}} \oplus A_{\bar{1}}$ over $\Bbbk$. A Lie superalgebra is a superalgerba $\mathfrak{g} = \mathfrak{g}_{\bar{0}} \oplus \mathfrak{g}_{\bar{1}}$ with the bilinear bracket (the super Lie bracket) $[\cdot, \cdot]: \mathfrak{g} \times \mathfrak{g} \to \mathfrak{g}$ which satisfies the following axioms, with homogeneous $x, y, z \in \mathfrak{g}$:
\[ [x, y] = - (-1)^{|x| |y|} [y, x], \]
\[ [x,[y,z]] = [[x,y],z] + (-1)^{|x| |y|} [y, [x,z]]. \]
A Lie superbialgebra $(\mathfrak{g}, [\cdot, \cdot], \delta)$ (see \cite{GZB91}, \cite{K06}) is a Lie superalgebra $(\mathfrak{g}, [\cdot, \cdot])$ with a skew-symmetric linear map $\delta: \mathfrak{g} \to \mathfrak{g} \otimes \mathfrak{g}$ that preserves the $\mathbb{Z}_{2}$-grading and satisfies the following conditions:
\begin{equation}
	\label{eq:bisuli1}
	( \delta \otimes id_{\mathfrak{g}} ) \circ \delta - (id_{\mathfrak{g}} \otimes \delta) \circ \delta = (id_{\mathfrak{g}} \otimes \tau_{\mathfrak{g}, \mathfrak{g}}) \circ (\delta \otimes id) \circ \delta,
\end{equation}
\begin{equation}
	\label{eq:bisuli2}
	\delta([x,y]) = ( ad_{x} \otimes id_{\mathfrak{g}} + id_{\mathfrak{g}} \otimes ad_{x} ) \delta(y) - ( ad_{y} \otimes id_{\mathfrak{g}} + id_{\mathfrak{g}} \otimes ad_{y} ) \delta(x),
\end{equation}
where $x, y \in \mathfrak{g}$, $id_{\mathfrak{g}}$ is the identity map on $\mathfrak{g}$, $ad_{x} y = [x,y]$ is the adjoint and $\tau_{V, W}: V \otimes W \to W \otimes V$ is the linear function given by
\begin{equation}
	\label{eq:taudef}
	\tau_{V, W}(v \otimes w) = (-1)^{|v| |w|} w \otimes v
\end{equation}
for homogeneous $v \in V$ and $w \in W$.

We use the well-known result (for more detail see \cite{Y91}, \cite{Y94}).

\begin{proposition}
	\label{pr:supliealgdef}
	Let $\mathfrak{g}$ be a Lie superalgebra of type A with associated Cartan matrix $(A = (a_{ij})_{i,j \in I}, \tau)$, where $\tau$ is a subset of $I = \{1,2,...,n\}$. Then $\mathfrak{g}$ is generated by $h_{i}$, $e_{i}$ and $f_{i}$ for $i \in I$ (whose parities are all even except for $e_{t}$ and $f_{t}$, $t \in \tau$, which are odd), where the generators satisfy the relations
	\[ [h_{i}, h_{j}] = 0, \; [h_{i}, e_{j}] = a_{ij} e_{j}, \; [h_{i}, f_{j}] = - a_{ij} f_{j}, \; [e_{i}, f_{j}] = \delta_{ij} h_{i} \]
	and the "super classical Serre-type" relations
	\[ [e_{i}, e_{i}] = [f_{i}, f_{i}] = 0, \; \text{if} \; i \in \tau, \]
	\[ (ad_{e_{i}})^{1+|a_{ij}|} e_{j} = (ad_{f_{i}})^{1+|a_{ij}|} f_{j} = 0, \; \text{if} \; i \ne j, \; \text{and} \; i \notin \tau, \]
	\[ [ [ [e_{m-1}, e_{m}], e_{m+1} ] , e_{m} ] = [ [ [f_{m-1}, f_{m}], f_{m+1} ] , f_{m} ] = 0, \; \text{if} \; m-1, m, m + 1 \in I \; \text{and} \; m \in \tau, \]
	where for $x \in \mathfrak{g}$ the linear mapping $ad_{x}: \mathfrak{g} \to \mathfrak{g}$ is defined by $ad_{x}(y) = [x,y]$ for all $y \in \mathfrak{g}$.
\end{proposition}

Denote by $\mathfrak{n}^{+}$ (resp. $\mathfrak{n}^{-}$) and $\mathfrak{h}$ the subalgebra of $\mathfrak{g}(A, \tau)$ generated by $e_{1}$, $...$, $e_{n}$ (resp. $f_{1}$, $...$, $f_{n}$) and $h_{1}$, $...$, $h_{n}$. Then define by $\mathfrak{b}^{+} = \mathfrak{h} \oplus \mathfrak{n}^{+}$ (resp. $\mathfrak{b}^{-} = \mathfrak{h} \oplus \mathfrak{n}^{-}$) the positive Borel subalgebra (resp. the negative Borel subalgebra) of $\mathfrak{g}(A, \tau)$.

We remind some categorical definitions. Our notations here follow \cite{BE17a} (see also \cite{BE17b}). Let $\mathcal{SV}ec$ denote the category of superspaces and all (not necessarily homogeneous) linear maps. Set $\underline{\mathcal{SV}ec}$ to be the subcategory of $\mathcal{SV}ec$ consisting of all superspaces but only the even linear maps (superspace morphisms). The tensor product equips $\underline{\mathcal{SV}ec}$ with a monoidal structure, and the map $u \otimes v \to (-1)^{|u| |v|} v \otimes u$ makes $\underline{\mathcal{SV}ec}$ into a strict symmetric monoidal category.

\begin{definition}
	\label{def:scsmcSL}
	\hspace{1em}
	\normalfont
	\begin{enumerate}
	\item A supercategory means a category enriched in $\underline{\mathcal{SV}ec}$, i. e. each morphism space is a superspace and composition induces an even linear map. A superfunctor between categories is a $\underline{\mathcal{SV}ec}$-enriched functor, i. e. a functor $F: \mathcal{A} \to \mathcal{B}$ such that the function $\text{Hom}_{\mathcal{A}}(\lambda, \mu) \to \text{Hom}_{\mathcal{B}}(F \lambda, F \mu)$, $f \to F f$ is an even linear map for all $ \lambda, \mu \in \text{Obj}(\mathcal{B})$.
	
	\item For any supercategory $\mathcal{A}$, the underlying category $\underline{\mathcal{A}}$ is the category with the same objects as $\mathcal{A}$ but only its even morphisms.
	
	\item Let $\text{sLieAlg}$ be the supercategory which objects are Lie superalgebras over field $\Bbbk$. A morphism $f \in \text{Hom}_{\text{sLieAlg}}(V,W)$ between Lie superalgebras $(V,[\cdot,\cdot]_{V})$ and $(W,[\cdot,\cdot]_{W})$ is a linear map of the underlying vector spaces such that $f([x,y]_{V}) = [f(x),f(y)]_{W}$ for all $x,y \in V$.
	
	\item Let $\text{sBiLieAlg}$ be the supercategory which objects are Lie superbialgebras over field $\Bbbk$. A morphism $f \in \text{Hom}_{\text{sLieAlg}}(V,W)$ between Lie superbialgebras $(V, [\cdot,\cdot]_{V}, \delta_{V})$ and $(W, [\cdot,\cdot]_{W}, \delta_{W})$ is a linear map of the underlying vector spaces such that $f([x,y]_{V}) = [f(x),f(y)]_{W}$ and $(f \otimes f) \circ \delta_{V}(x) = \delta_{W} \circ f(x)$ for all $x,y \in V$.
	
	\end{enumerate}
\end{definition}

We also need the following general result.

\begin{proposition}
	\label{pr:cobraindisomLiebi}
	Let $f \in \text{Hom}_{\underline{\text{sLieAlg}}}(\mathfrak{g_1}, \mathfrak{g_2})$ be an isomorphism. Suppose that $\mathfrak{g_1}$ is a Lie superbialgebra with a skew-symmetric even linear map $\delta_{\mathfrak{g_1}}: \mathfrak{g_{1}} \to \mathfrak{g_{1}} \otimes \mathfrak{g_{1}}$ which satisfies \eqref{eq:bisuli1} - \eqref{eq:bisuli2}. Then $f$ induces a Lie superbialgebra structure on $\mathfrak{g_2}$, where a skew-symmetric even linear map $\delta_{\mathfrak{g_{2}}}: \mathfrak{g_{2}} \to \mathfrak{g_{2}} \otimes \mathfrak{g_2}$ which satisfies \eqref{eq:bisuli1} - \eqref{eq:bisuli2} is defined by
	\[ \delta_{\mathfrak{g_{2}}} := (f \otimes f) \circ \delta_{\mathfrak{g_{1}}} \circ f^{-1}. \]
\end{proposition}

The special Lie superalgebra $sl(2|1)$ over $\Bbbk$ is the algebra $M_{3,3}(\Bbbk)$ of $3 \times 3$ matrices over $\Bbbk$, $\mathbb{Z}_2$-graded as $sl(2|1)_{\bar{0}} \oplus sl(2|1)_{\bar{1}}$, where
\[ sl(2|1)_{\bar{0}} = \{X = diag(A, D) | Str(X) := tr(A) - tr(D) = 0, \; A \in M_{2,2}(\Bbbk), \; D \in M_{1,1}(\Bbbk) \}, \]
and
\[ sl(2|1)_{\bar{1}} = \{ \begin{psmallmatrix} 0 & B \\ C & 0 \end{psmallmatrix} | B \in M_{2,1}(\Bbbk), \; C \in M_{1,2}(\Bbbk) \}, \]
with the bilinear super bracket $[x, y] = xy - (-1)^{a b} y x$ for $x \in sl(2|1)_{\bar{a}}, \; y \in sl(2|1)_{\bar{b}}, \; \bar{a}, \bar{b} \in \mathbb{Z}_2$,  on $sl(2|1)$.

We choose for basis of Lie superalgebra $sl(2|1)$ over $\Bbbk$ the following elements:
$ h_1=e_{1,1}-e_{2,2}, \quad h_2=e_{2,2}+e_{3,3}, $
$ e_1=e_{1,2}, \quad f_1=e_{2,1},$
$ e_2=e_{2,3}, \quad f_2=e_{3,2}, $
$ e_3=[e_1,e_2]=e_{1,3}, \quad f_3=[f_1,f_2]=-e_{3,1}, $
where $e_{i,j} \in M_{3,3}(\Bbbk)$ denotes matrix with $1$ at $(i,j)$-position and zeros elsewhere.
The elements $h_1, h_2, e_1, f_1$ are even and $e_2, f_2, e_3, f_3$ are odd.
We have $[h_i,h_j] = 0$, $[h_i, e_j] = a_{ij} e_j$, $[h_i,f_j]=-a_{ij} f_j$, $[e_i,f_j]=  \delta_{ij} h_i$,
$[e_2,e_2]=[f_2,f_2]=0$, $[e_1,[e_1,e_2]]=[f_1,[f_1,f_2]]=0$
with $(a_{ij})$ the matrix $$
A=\left(
\begin{array}{cc}
2 & -1 \\
-1 & 0 \\
\end{array}
\right).
$$

The Cartan subalgebra of $sl(2|1)$ is the $\Bbbk$-span $\mathfrak{h} = \langle h_1,h_2 \rangle$. Denote by $\mathfrak{h}^{*}$ the dual space of $\mathfrak{h}$. $sl(2|1)$ decomposes as a direct sum of root spaces $\mathfrak{h} \oplus \bigoplus_{\alpha \in \mathfrak{h}^{*}} sl(2|1)_{\alpha}$, where
\[ sl(2|1)_{\alpha} = \{ X \; | \; [h,X] = \alpha(h) X, \; \forall h \in \mathfrak{h} \}. \]
An $\alpha \in \mathfrak{h}^{*} - \{0\}$ is called a root if the root space $sl(2|1)_{\alpha}$ is not zero. The root system for $sl(2|1)$ is defined to be $\Delta = \{ \alpha \in \mathfrak{h}^{*} \; | \; sl(2|1)_{\alpha} \ne 0, \alpha \ne 0 \}$. Define sets of even and odd roots, respectively, to be $\Delta_{\bar{0}} = \{ \alpha \in \Delta \; | \; sl(2|1)_{\alpha} \cap sl(2|1)_{\bar{0}} \ne 0 \}$, $\Delta_{\bar{1}} = \{ \alpha \in \Delta \; | \; sl(2|1)_{\alpha} \cap sl(2|1)_{\bar{1}} \ne 0 \}$. Thus we can define a parity function $\lvert \cdot \rvert_{\Delta}: \Delta \to \mathbb{Z}_2$ by $|x|_{\Delta} = \bar{a}$ if $x \in \Delta_{\bar{a}}$, where $\bar{a} \in \mathbb{Z}_2$.

Consider the $\Bbbk$-span $\mathfrak{d} = \langle e_{11}, e_{22}, e_{33} \rangle$ and it's dual space $\mathfrak{d}^{*} = \langle \epsilon_{1}, \epsilon_{2}, \delta_{1} \rangle$. We define a non-degenerate symmetric bilinear form $(\cdot, \cdot) :\mathfrak{d}^{*} \times \mathfrak{d}^{*} \to \Bbbk $ by
\[ (\epsilon_{i}, \epsilon_{j}) = \delta_{ij}, \; (\epsilon_{i}, \delta_{1}) = 0, \; (\delta_{1}, \delta_{1}) = -1 \]
for all $i, j \in I$, where $I := \{1,2\}$.

Notice that $\mathfrak{h}^{*} \subset \mathfrak{d}^{*}$. Then the root system $\Delta \subseteq \mathfrak{h}^{*}$ has the form $\Delta = \Delta_{\bar{0}} \oplus \Delta_{\bar{1}}$, where
$ \Delta_{\bar{0}}=\{\pm(\epsilon_1-\epsilon_2)\}, $ $ \Delta_{\bar{1}}=\{ \pm(\epsilon_1-\delta_1),  \pm(\epsilon_2-\delta_1) \}$. Accordingly, we also have the decomposition $\Delta = \Delta^{+} \cup \Delta^{-}$, where $\Delta^{+} = \{\epsilon_1-\epsilon_2, \epsilon_1-\delta_1, \epsilon_2-\delta_1 \}$ and $\Delta^{-} = \{\epsilon_2-\epsilon_1, \delta_1-\epsilon_1, \delta_1-\epsilon_2 \}$. We choose the basis $\tau = \{ \alpha_{1} := \epsilon_{1} - \epsilon_{2}, \; \alpha_{2} := \epsilon_2-\delta_1 \}$. The form $(\cdot, \cdot)$ on $\mathfrak{d}^{*}$ induces a non-degenerate symmetric bilinear form on $\mathfrak{h}^{*}$, which will be denoted by $(\cdot, \cdot)$ as well. We define a natural pairing $\langle \cdot , \cdot \rangle : \mathfrak{h} \times \mathfrak{h}^{*} \to \Bbbk $ by linearity with $\langle h_i, \alpha \rangle = \alpha (h_i)$ for all $i \in I, \; \alpha \in \tau$.

The Cartan matrix is $A = (a_{ij} = \alpha_{j} (h_i); \; \alpha_{j} \in \tau, i,j \in I)$. One could describe $A$ by the corresponding Dynkin diagram. Join vertex $i$ with vertex $j$ if $a_{ij} \ne 0$. We need two types of vertices: $\circ$ if $a_{ii} = 2$ and $|\alpha_{i}|_{\Delta} = 0$ (it is called white dot); $\otimes$ if $a_{ii} = 0$ and $|\alpha_{i}|_{\Delta} = 1$ (it is called grey dot), where $i \in I$. We say that two Dynkin diagrams or even graphs that correspond to them are isomorphic if they have the equal number of white and grey dots.

Define the linear function $\delta_{sl(2|1)}: sl(2|1) \to sl(2|1) \otimes sl(2|1)$ on the generators by
\begin{equation}
	\label{eq:LSBstr}
	\delta_{sl(2|1)}(h_{i}) = 0, \; \delta_{sl(2|1)}(e_{i}) = \frac{1}{2} (h_{i} \otimes e_{i} - e_{i} \otimes h_{i}), \; \delta_{sl(2|1)}(f_{i}) = \frac{1}{2} (h_{i} \otimes f_{i} - f_{i} \otimes h_{i})
\end{equation}
for $i \in I$, and extend it to all the elements of $sl(2|1)$ using equation \eqref{eq:bisuli2} and by linearity. Then $sl(2|1)$ becomes a Lie superbialgebra.

\section{Weyl groupoid}
\label{sec:WG}

We give a categorical definition of a Weyl groupoid. This enable us to describe the Weyl groupoid of $sl(2|1)$ by generators and relations. We mention how it is connected with the supercategory of Lie superalgebras.

\subsection{Cartan schemes and definition of Weyl groupoid}
\label{WGD}
We adopt to our purposes the definition of a Weyl groupoid which was introduced in \cite{HY08} and reformulated in \cite{CH09} (see also \cite{S11}, \cite{H09}, \cite{LS91}, \cite{K90}). Thus we define Weyl groupoid as a supercategory. In Section \ref{sc:Wgsl} we give the example how Lie superalgebra $sl(2|1)$ fits in our definition.

In order to define Weyl groupoid we need auxiliary data. In this way we associate with an object (Dynkin diagram) of the Weyl groupoid an element of a set $D$ which labels it, root basis $\tau$, maps $\rho$ which indicate the direction of the action and special matrices $C$ used to define reflections.

Let $V$ be a fixed $(N_{\bar{0}}, N_{\bar{1}})$-dimensional super vector space and $N = N_{\bar{0}} + N_{\bar{1}}$.

\begin{definition}
	\normalfont
	\label{def:CSch}
	Let $A$ and $D$ be non-empty sets, where $A = (a_{d})_{d \in D}$, $\tau^{d}$ be a basis of $V$ for all ${d \in D}$,  $\rho_{\alpha} : S_{d} \rightarrow S_{d^{'}}$ bijective  mappings for all $\alpha \in \tau^{d}$, $d \in D$ and for some non-empty $S_{d}, S_{d^{'}} \subseteq A$, $d^{'} \in D$, and $N \times N$ matrices $C^{d} = (c_{\alpha,\beta}^{d} \in \mathbb{Z})_{\alpha,\beta \in \tau^d}$ for all $d \in D$. The tuple
	\[ \mathcal{C} = \mathcal{C}(A,D,V,(\tau^d)_{d \in D}, (\rho^d_{\alpha})_{\alpha \in \tau^d, d \in D}, (C^{d})_{d \in D}) \]
	is called a Cartan scheme if for all $d \in D$
	\begin{enumerate}
		\item if $\rho_{\alpha}: S_{d} \to S_{d^{'}}$ then $\exists \beta \in S_{d^{'}}$:  $\rho_{\beta} \rho_{\alpha} = id_{S_{d}}$ and $\rho_{\alpha} \rho_{\beta} = id_{S_{d^{'}}}$ for all $\alpha \in \tau_d$, $d \in D$ and for some $d^{'} \in D$,
		\item $c_{\alpha,\alpha}^d = 2$ and $c_{\alpha,\beta}^d \le 0$, where $\alpha,\beta \in \tau^d$ with $\alpha \ne \beta$,
		\item if $c_{\alpha,\beta}^d = 0$, then $c_{\beta,\alpha}^d = 0$, where $\alpha, \beta \in \tau^d$,
		\item \label{def:CSch4} if $\rho_{\alpha}: S_{d} \to S_{d^{'}}$ then there exists bijective mapping $i: \tau^d \to \tau^{d'}$ such that  $c^d_{\alpha, \beta} = c^{d'}_{i(\alpha), i(\beta)}$ for all $\alpha, \beta \in \tau_d$, $d \in D$ and for some $d^{'} \in D$.
	\end{enumerate}
\end{definition}

Now we are able to formulate the definition of a Weyl groupoid where morphisms are generalizations of compositions of reflections.

\begin{definition}
	\label{def:WG}
	\normalfont
	Let $\mathcal{C} = \mathcal{C}(A,D,V,(\tau^d)_{d \in D}, (\rho_{\alpha})_{\alpha \in \tau^d, d \in D}, (C^{d})_{d \in D})$ be a Cartan scheme. For all $d \in D$ and $\alpha, \beta \in \tau^d$ define the function $\sigma_{\alpha}^{d} \in  GL_{N}(V)$ by
	\begin{equation}
	\label{eq:genmorroot}
	\sigma_{\alpha}^{d}(\beta) = \beta - c_{\alpha,\beta}^{d} \alpha.
	\end{equation}
	The Weyl groupoid of $\mathcal{C}$ is the supercategory $\mathcal{W}(\mathcal{C})$ such that $\text{Obj}(\mathcal{W}(\mathcal{C})) = A$ and the morphisms are compositions of maps $\sigma_{\alpha}^{d}$ with $d \in D$ and $\alpha \in \tau^d$, where $\sigma_{\alpha}^{d}$ is considered as an element of super vector space $\text{Hom}_{\mathcal{W}(\mathcal{C})}(a_d,\rho_{\alpha}(a_d))$. The cardinality of $D$ is the rank of $\mathcal{W}(\mathcal{C})$.
\end{definition}

It is easy to see that the above definition is correct and the Weyl groupoid is in fact a supercategory, where the parity of morphism $\sigma_{\alpha}^{d}$ is determined by parity of $\alpha$ as an element of a basis $\tau^{d}$, where $d \in D$. We reformulate the condition \ref{def:CSch4} in definition \ref{def:CSch} in the following way $c^d_{\alpha, \beta} = c^{d^{'}}_{\sigma_{\alpha}^{d}(\alpha), \sigma_{\alpha}^{d}(\beta)}$, where $\rho_{\alpha}(a_d) = a_{d^{'}}$ ($d^{'} \in D$) and for all $\alpha, \beta \in \tau_d$, $d \in D$. Thus a morhpism $\sigma_{\alpha}^{d}$ \eqref{eq:genmorroot} becomes invertible and the inverse is $\sigma_{\sigma_{\alpha}^{d}(\alpha)}^{d^{'}}$.

\begin{definition}
	\normalfont
	A Cartan scheme is called connected if its Weyl groupoid is connected, that is, if for all $a,b \in A$ there exists $w \in \text{Hom}_{\mathcal{W}(\mathcal{C})}(a,b)$. The Cartan scheme is called simply connected, if it is connected and \[\text{Hom}_{\mathcal{W}(\mathcal{C})}(a,a) = \{id_{a}\}\]
	for all $a \in A$.
\end{definition}

We characterize root systems in axiomatic way and also add explicit conditions that are imposed on reflections.

\begin{definition}
	\normalfont
	Let $\mathcal{C} = \mathcal{C}(A,D,V,(\tau^d)_{d \in D}, (\rho^d_{\alpha})_{\alpha \in \tau^d, d \in D}, (C^{d})_{d \in D})$ be a Cartan scheme. For all $a_d \in A$ let $R^{a_d} \subseteq V$, and define $m_{\alpha,\beta}^{a_d} = |R^{a_d} \cap (\mathbb{N}_{0} \alpha + \mathbb{N}_{0} \beta)|$ for all $\alpha, \beta \in \tau^{d}$ and $d \in D$. We say that
	\[ \mathcal{R} = \mathcal{R}(\mathcal{C}, (R^{a})_{a \in A}) \]
	is a root system of type $\mathcal{C}$, if it satisfies the following axioms:
	\begin{enumerate}
		\item exists decomposition $R^{a} = R_{+}^{a} \cup -R_{+}^{a}$, for all $a \in A$;
		\item $R^{a_d} \cap \mathbb{Z} \alpha = \{\alpha, -\alpha\}$ for all $\alpha \in \tau^d$ and $d \in D$;
		\item $\sigma_{\alpha}^{d}(R^{a_d}) = R^{\rho_{\alpha}^{d}(a_d)}$ for all $\alpha \in \tau^d$ and $d \in D$;
		\item for $a, b \in A$, if $id_{V} \in \text{Hom}_{\mathcal{W}(\mathcal{C})}(a,b)$, then $a = b$.
	\end{enumerate}
\end{definition}

The elements of the set $R^{a}$, where $a \in A$, are called roots. The root system $\mathcal{R}$ is called finite if for all $a \in A$ the set $R^{a}$ is finite. If $\mathcal{R}$ is a root system of type $\mathcal{C}$, then we say that $\mathcal{W}(\mathcal{R}) := \mathcal{W}(\mathcal{C})$ is the Weyl groupoid of $\mathcal{R}$.

\subsection{Weyl groupoid of $sl(m|n)$}
\label{sc:Wgsl}

Now we are able to construct the Weyl groupoid of the Lie superalgebra $sl(2|1)$. Note that it follows from \cite{AYY} that Definition \ref{def:WG} of a Weyl groupoid is true for the simple Lie superalgebra $A(m,n)$ for $m,n \in \mathbb{Z}_{\ge 0}$, where $A(m,n)$ is defined by $sl(m+1,n+1)$ if $m \ne n$, and otherwise $A(n,n) := sl(n+1,n+1) / \mathfrak{i}$, where $\mathfrak{i}$ is its unique one-dimensional ideal.  We use notations $\epsilon_i$ ($i \in \overline{1,m}$), $\delta_{j-m}$ ($j \in \overline{m+1,m+n}$) from Sections \ref{subs:dsl} and \ref{WGD}.

Let $I = \overline{1,m+n-1}$. The elements of a finite set $D \subset \mathbb{N}$ will be used to label different Dynkin diagrams for $A(m-1,n-1)$. We make a convention to parameterize elements of the dual space $\mathfrak{d}^{*}$ by the set $I(m|n) = \{1,2,...,m+n\}$, where
\[ 	\bar{\epsilon}_{i} :=
\begin{cases}
	\epsilon_{i},& \mbox{ if } \; i \in \overline{1,m}, \\
	\delta_{i-m},& \mbox{ if } \; i \in \overline{m+1,m+n}.
\end{cases} \]
Let
\begin{equation}
	\label{eq:rootdecomp}
	(\tau^{d} = \{\alpha_{1,d}= \bar{\epsilon}_{i_1} - \bar{\epsilon}_{i_2}, \; \alpha_{2,d} = \bar{\epsilon}_{i_2} - \bar{\epsilon}_{i_3} \; , ... , \; \alpha_{m+n-1,d} = \bar{\epsilon}_{i_{m+n-1}} - \bar{\epsilon}_{i_{m+n}}  |  \{i_{1}, ... , i_{m+n}\} = I(m|n) \})_{d \in D}.	
\end{equation}
Notice that $\tau^{d}$ is the basis of $\mathfrak{h}^{*}$ for all ${d \in D}$. 

Set
\[ \tau^{1} = \{ \alpha_{1,1} = \epsilon_{1} - \epsilon_{2}, \alpha_{2,1} = \epsilon_{2} - \epsilon_{3}, ... , \alpha_{m,1} = \epsilon_{m} - \delta_{1}, \alpha_{m+1,1} = \delta_{1} - \delta_{2} , ... , \alpha_{m+n-1,1} = \delta_{n-1} - \delta_{n} \}. \]

We recall  that $\tau^{1}$ is called a distinguished simple root system for a Lie superalgebra of type $sl(m,n)$ ($A(m-1,n-1)$) (\cite{FSS00}).

Consider the family of symmetric matrices $A_{d} = ((\alpha_{i}^{d}, \alpha_{j}^{d}))_{i,j \in I}$, where the non-degenerate symmetric bilinear form $(\cdot, \cdot) :\mathfrak{h}^{*} \times \mathfrak{h}^{*} \to \Bbbk $ is defined as in Section \ref{subs:dsl}. Define a family of tuples $A = (a_{d} = (A_{d}, \tau^{d}) )_{d \in D}$.

Let $c_{\alpha,\beta} := -\text{max}\{k \in \mathbb{Z}_{\ge 0} \; | \; \beta + k \alpha \in \Delta\}$ for $\alpha \ne \beta$, and $c_{\alpha,\alpha} := 2$, where $\alpha, \beta \in \Delta$. Introduce a family of matrices $(C^{d} = (c_{\alpha,\beta}^{d} := c_{\alpha,\beta})_{\alpha, \beta \in \tau^d} )_{d \in D}$.

Denote the (usual) left action of the symmetric group $S_{m+n}$ on $I(m|n)$ by $\triangleright: S_{m+n} \times I(m|n) \to I(m|n)$ and on $\Delta$ by $\circlearrowleft: S_{m+n} \times \Delta \to \Delta$, where $s \circlearrowleft (\bar{\epsilon}_{j_1} - \bar{\epsilon}_{j_2}) = \bar{\epsilon}_{s \triangleright j_1} - \bar{\epsilon}_{s \triangleright j_2}$, for $s \in S_{m+n}, \; j_1, j_2 \in I(m|n)$. Thus define for all $\alpha \in \tau^d$ ($d \in D$) functions $\rho_{\alpha}^{d}: \{ a_{d} \} \to \{ a_{b} \}$, such that $\rho_{\alpha}^{d}(a_d) = a_{b}$, where $b \in D$, $\alpha = \bar{\epsilon}_{j_1} - \bar{\epsilon}_{j_2}$ and $\tau^{b} = \{ \alpha_{k}^{'} := (j_1,j_2) \circlearrowleft \alpha_{k} = \sigma_{\alpha}^{d}(\alpha_{k}) \; | \; \alpha_{k} \in \tau^{d}, k \in I \}$.

Consider the simply connected Cartan scheme
\[\mathcal{C} = \mathcal{C}(A,D,\mathfrak{h}^{*},(\tau^d)_{d \in D}, (\rho^d_{\alpha})_{\alpha \in \tau^d, d \in D}, (C^{d})_{d \in D}). \]
Notice that $\mathcal{R} = \mathcal{R}(\mathcal{C}, (\Delta=\Delta^{a})_{a \in A})$ is the root system of type $\mathcal{C}$. We call $\mathcal{W}(\mathcal{R})$ the Weyl groupoid of $sl(m|n)$. It is easy to construct Weyl groupoid for $sl(2|1)$ (see Fig. \ref{pict:DD}).

\begin{center}
	\begin{tikzpicture}
	\node[scale=1.4]{
		\begin{tikzcd}[sep = normal, column sep=normal]
		\underset{\epsilon_1-\epsilon_2}{\bigcirc} \arrow[dash, shorten <= -.95em, shorten >= -.9em, "d=1"{name=U1}]{r} & \underset{\epsilon_2-\delta_1}{\bigotimes} \ar[r, shorten <= -.5em, shorten >= -.5em, shift right=-0.5ex,"\sigma_{\epsilon_2-\delta_1}^{1}"] & \ar[l, shorten <= -.5em, shorten >= -.5em, shift right=-0.5ex,"\sigma_{\delta_1-\epsilon_2}^{3}"] \underset{\epsilon_1-\delta_1}{\bigotimes} \arrow[dash, shorten <= -.95em, shorten >= -.9em, "d=3"{name=U2}]{r} & \underset{\delta_1-\epsilon_2}{\bigotimes} \ar[r, shorten <= -.5em, shorten >= -.5em, shift right=-0.5ex,"\sigma_{\epsilon_1-\delta_1}^{3}"] & \ar[l, shorten <= -.5em, shorten >= -.5em, shift right=-0.5ex,"\sigma_{\delta_1-\epsilon_1}^{5}"] \underset{\delta_1-\epsilon_1}{\bigotimes} \arrow[dash, shorten <= -.95em, shorten >= -.9em, "d=5"{name=U3}]{r} & \underset{\epsilon_1-\epsilon_2}{\bigcirc} \\
		\underset{\epsilon_2-\epsilon_1}{\bigcirc} \arrow[dash, shorten <= -.95em, shorten >= -.9em, "d=2"{name=D1, below}]{r} & \underset{\epsilon_1-\delta_1}{\bigotimes} \ar[r, shorten <= -.5em, shorten >= -.5em, shift right=-0.5ex,"\sigma_{\epsilon_1-\delta_1}^{2}"] & \ar[l, shorten <= -.5em, shorten >= -.5em, shift right=-0.5ex,"\sigma_{\delta_1-\epsilon_1}^{4}"] \underset{\epsilon_2-\delta_1}{\bigotimes} \arrow[dash, shorten <= -.95em, shorten >= -.9em, "d=4"{name=D2, below}]{r} & \underset{\delta_1-\epsilon_1}{\bigotimes} \ar[r, shorten <= -.5em, shorten >= -.5em, shift right=-0.5ex,"\sigma_{\epsilon_2-\delta_1}^{4}"] & \underset{\delta_1-\epsilon_2}{\bigotimes} \ar[l, shorten <= -.5em, shorten >= -.5em, shift right=-0.5ex,"\sigma_{\delta_1-\epsilon_2}^{6}"] \arrow[dash, shorten <= -.95em, shorten >= -.9em, "d=6"{name=D3, below}]{r} & \underset{\epsilon_2-\epsilon_1}{\bigcirc}
		
		\arrow[d, shorten <= .6em, shorten >= .8em, from=U1, to=D1, shift right=0.5ex,swap,"\sigma_{\epsilon_1-\epsilon_2}^{1}"]
		\arrow[u, shorten <= .8em, shorten >= .6em, from=D1, to=U1, shift right=0.5ex,swap,"\sigma_{\epsilon_2-\epsilon_1}^{2}"]
		\arrow[d, shorten <= .6em, shorten >= .8em, from=U3, to=D3, shift right=0.5ex,swap,"\sigma_{\epsilon_1-\epsilon_2}^{5}"]
		\arrow[u, shorten <= .8em, shorten >= .6em, from=D3, to=U3, shift right=0.5ex,swap,"\sigma_{\epsilon_2-\epsilon_1}^{6}"]
		
		\end{tikzcd}
	};
	\end{tikzpicture}
	
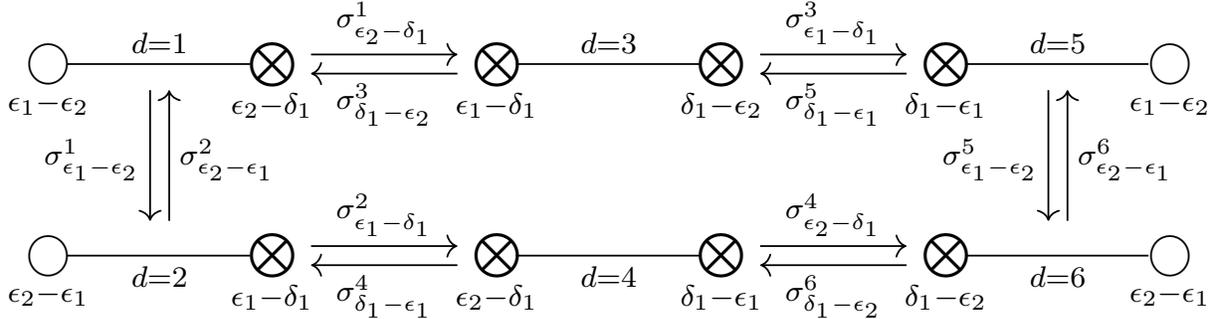
\captionof{figure}{\textbf{Dynkin Diagrams of $sl(2|1)$}}
	\label{pict:DD}
\end{center}

$\mathcal{W}(\mathcal{R})$ is the supercategory generated by morphisms (recall \eqref{eq:genmorroot})
\begin{equation}
	\label{eq:WCgenel}
	\mathcal{B} = \{\sigma_{\alpha_{i,d}}^{d} \in \text{Hom}(\mathcal{W}(\mathcal{R})) | \; i \in I , d \in D \}
\end{equation}
and by relations for $j \in J$
\begin{equation}
	\label{eq:invWCeqgenl}
	\sigma_{x}^{d_2} \sigma_{\alpha_{i,d_1}}^{d_1} = id_{a_{d_1}}, \; \sigma_{\alpha_{i,d_1}}^{d_1} \sigma_{x}^{d_2} = id_{a_{d_2}},
\end{equation}
where $x = \sigma_{\alpha_{i,d_1}}^{d_1}(\alpha_{i,d_1})$;
\begin{equation}
 \sigma_{\alpha_{j,d_2}}^{d_2} \sigma_{\alpha_{i,d_1}}^{d_1} =  \sigma_{\alpha_{i,d_2}}^{d_2} \sigma_{\alpha_{j,d_1}}^{d_1},
\end{equation}
iff $|i-j| \ge 2$ and $|\alpha_{i,d_1}| = |\alpha_{j,d_1}| = 0$;
\begin{equation}
	\sigma_{\alpha_{j,d_4}}^{d_4} \sigma_{\alpha_{i,d_3}}^{d_3} \sigma_{\alpha_{j,d_2}}^{d_2} \sigma_{\alpha_{i,d_1}}^{d_1} = \sigma_{\alpha_{i,d_4}}^{d_4} \sigma_{\alpha_{j,d_3}}^{d_3} \sigma_{\alpha_{i,d_2}}^{d_2} \sigma_{\alpha_{j,d_1}}^{d_1},
\end{equation}
iff $|i-j| \ge 2$;
\begin{equation}
	\sigma_{\alpha_{i,d_3}}^{d_3} \sigma_{\alpha_{j,d_2}}^{d_2} \sigma_{\alpha_{i,d_1}}^{d_1} =  \sigma_{\alpha_{j,d_3}}^{d_3} \sigma_{\alpha_{i,d_2}}^{d_2} \sigma_{\alpha_{j,d_1}}^{d_1},
\end{equation}
iff $|i-j| = 1$ and $|\alpha_{i,d_1}| = |\alpha_{j,d_1}|$;
\begin{equation}
	\label{eq:WCeqgenl}
	\sigma_{\alpha_{j,d_6}}^{d_6} \sigma_{\alpha_{i,d_5}}^{d_5} \sigma_{\alpha_{j,d_4}}^{d_4} \sigma_{\alpha_{i,d_3}}^{d_3} \sigma_{\alpha_{j,d_2}}^{d_2} \sigma_{\alpha_{i,d_1}}^{d_1} = \sigma_{\alpha_{i,d_6}}^{d_6} \sigma_{\alpha_{j,d_5}}^{d_5} \sigma_{\alpha_{i,d_4}}^{d_4} \sigma_{\alpha_{j,d_3}}^{d_3} \sigma_{\alpha_{i,d_2}}^{d_2} \sigma_{\alpha_{j,d_1}}^{d_1},
\end{equation}
iff $|i-j| = 1$.

Thus we have theorem (definition)

\begin{theorem}
	\label{th:bga}
	Weyl groupoid $\mathcal{W}(\mathcal{R})$ has the structure of the braid groupoid of type $A$, i. e. it is groupoid with generators $\mathcal{B}$ \eqref{eq:WCgenel}  and relations $\sigma_{\alpha} \sigma_{\beta} ... = \sigma_{\beta} \sigma_{\alpha} ...$ for all $\alpha, \beta \in \mathcal{B}$, $\alpha \ne \beta$ with $m(\alpha,\beta)$ factors on both sides, where $m(\alpha, \beta)$ is the order of $\sigma_{\alpha} \sigma_{\beta}$ in $\mathcal{W}(\mathcal{R})$.
\end{theorem}
\begin{proof}
	The results follows immediately from equations \eqref{eq:invWCeqgenl} - \eqref{eq:WCeqgenl} and direct computations.
\end{proof}

Let us note that the above described braid groupoid is a natural generalization of the braid group.

\vspace{0.3cm}

\subsection{Connection with the supercategory of Lie superalgebras}
\label{sc:CcLs}

We use notations introduced in Section \ref{sc:Wgsl}. Fix $A_{d}=((\alpha_{i}^{d}, \alpha_{j}^{d}))_{i,j \in I}$ and the set of simple roots $\tau^{d}$ for $d \in D$. Define a Lie superalgebra $\mathfrak{g}(A_d, \tau^{d})$ to be a Lie superalgebra generated as in  Proposition \ref{pr:supliealgdef}. Recall the definition of the supercategory $\underline{\text{sLieAlg}}$ (see Definition \ref{def:scsmcSL}).

We are able to construct the covariant faithful functor $F: \mathcal{W}(\mathcal{R}) \to \underline{\text{sLieAlg}}$. Thus the action on objects is given by the formula
\begin{equation}
	\label{eq:funcsuplieobj}
	F((A_d, \tau^{d}))= \mathfrak{g}(A_d, \tau^{d}),
\end{equation}
where $A_{d}=((\alpha_{i}^{d}, \alpha_{j}^{d}))_{i,j \in I}$ and $d \in D$. Notice that $sl(m, n) = \mathfrak{g}(A_1, \tau^{1})$, where $\tau^{1}$ be a distinguished system of simple roots (\cite{FSS00}).

Let $\rho(a_{d_1}) = a_{d_2}$ for $d_1, d_2 \in D$. We prove in Theorem \ref{th:suplieclasisom} that $\mathfrak{g}(A_{d_1},\tau^{d_1})$ and $\mathfrak{g}(A_{d_2},\tau^{d_2})$ are isomorphic in $\underline{\text{sLieAlg}}$.  Consider a generator $\sigma_{\alpha_{i,d_1}}^{d_1} \in \text{Hom}_{\mathcal{W}(\mathcal{R})}(a_{d_1}, a_{d_2})$ \eqref{eq:WCgenel} for an appropriate $\alpha_{i,d_1} \in \tau^{d_1}$ ($i \in I$) and fix an isomorphism
\[ V_{i, d_1} \in \text{Hom}_{\underline{\text{sLieAlg}}}(\mathfrak{g}(A_{d_1},\tau^{d_1}),  \mathfrak{g}(A_{d_2},\tau^{d_2})). \]
Define $F(\sigma_{\alpha_{i,d_1}}^{d_1}) = V_{i, d_1}$ and $F(\sigma_{x}^{d_2}) = V_{i, d_2}$, where $x = \sigma_{\alpha_{i,d_1}}^{d_1}(\alpha_{i,d_1})$ and $V_{i, d_2} := V_{i, d_1}^{-1}$. Additionally suppose that isomorphisms $V_{i,d}$ satisfy relations \eqref{eq:invWCeqgenl} - \eqref{eq:WCeqgenl} for all $i \in I$ and $d \in D$. It is easy to see that $F$ is indeed the covariant faithful functor that preserves the structure of the Weyl groupoid. We give an example of the family of isomorphisms $\{ F(\sigma) \in \text{Hom}(\underline{\text{sLieAlg}}) \}_{\sigma \in \mathcal{B}}$. We formulate

\begin{theorem}
	\label{th:suplieclasisom}
	There exist the unique covariant faithful functor $F: \mathcal{W}(\mathcal{R}) \to \underline{\text{sLieAlg}}$ which satisfies equation \eqref{eq:funcsuplieobj} and for all $\sigma_{\alpha_{i,d_1}}^{d_1} \in \mathcal{B}$ ($i \in I$)
	\begin{equation}
		\label{eq:functsupliealgedefmor}
		F(\sigma_{\alpha_{i,d_1}}^{d_1}) = L_{i, d_1}, \; F(\sigma_{x}^{d_2}) = L_{i, d_2}^{-}
	\end{equation}
	and
	\[ L_{i, d_1}, L_{i, d_1}^{-}: \mathfrak{g}(A_{d_1},\tau^{d_1}) \to \mathfrak{g}(A_{d_2},\tau^{d_2}) \]
	are unique isomorphisms in $\underline{\text{sLieAlg}}$ satisfying equations \eqref{eq:cslif} - \eqref{eq:cslis} below for $|i-j| = 1$ and $|i-l| \ge 2$ ($j,l \in I$)
	\begin{equation}
	\label{eq:cslif}
	L_{i, d_1}(h_{i,d_1}) = L_{i, d_1}^{-}(h_{i,d_1}) = - h_{i,d_2},
	\end{equation}
	\begin{equation}
	L_{i, d_1}(h_{j,d_1}) = L_{i, d_1}^{-}(h_{j,d_1}) = h_{i,d_2} +  h_{j,d_2}, \; L_{i, d_1}(h_{l,d_1}) = L_{i, d_1}^{-}(h_{l,d_1}) = h_{l,d_2},
	\end{equation}
	\begin{equation}
	L_{i, d_1}(e_{i,d_1}) = (-1)^{|\alpha_{i,d_1}|} f_{i, d_2},
	\end{equation}
	\begin{equation}
	L_{i, d_1}(f_{i,d_1}) = e_{i, d_2},	
	\end{equation}
	\begin{equation}
	L_{i, d_1}^{-}(e_{i,d_1}) = f_{i,d_2},
	\end{equation}
	\begin{equation}
	L_{i, d_1}^{-}(f_{i,d_1}) = (-1)^{|\alpha_{i,d_1}|} e_{i,d_2},
	\end{equation}
	\begin{equation}
	L_{i, d_1}(e_{j,d_1}) = - (x,y) [ e_{i,d_2}, e_{j,d_2} ],		
	\end{equation}
	\begin{equation}
	L_{i, d_1}(f_{j,d_1}) = [ f_{j,d_2}, f_{i,d_2} ],
	\end{equation}
	\begin{equation}
	L_{i, d_1}^{-}(e_{j,d_1}) = (\alpha_{i,d_1},\alpha_{j,d_1}) (-1)^{|x|(|x|+|y|)+[b=-1]+[b(x,y)=1]} [e_{j,d_2}, e_{i,d_2}],
	\end{equation}
	\begin{equation}
	L_{i, d_1}^{-}(f_{j,d_1}) = (-1)^{1+|x||y|+[b=1]+[b(x,y)=-1]} [f_{i,d_2}, f_{j,d_2}],
	\end{equation}
	\begin{equation}
	L_{i, d_1}(e_{l,d_1}) = L_{i, d_1}^{-}(e_{l,d_1}) = e_{l,d_2},
	\end{equation}
	\begin{equation}
	\label{eq:cslis}
	L_{i, d_1}(f_{l,d_1}) = L_{i, d_1}^{-}(f_{l,d_1}) = f_{l,d_2},
	\end{equation}
	where $x =\sigma_{\alpha_{i,d_1}}^{d_1}(\alpha_{i,d_1})$, $y = \sigma_{\alpha_{i,d_1}}^{d_1}(\alpha_{j,d_1})$ and $b = [|\alpha_{i,d_1}| = 0] - [|\alpha_{i,d_1}| = 1]$.
	
	One has $L_{i, d_2}^{-} = (L_{i, d_1})^{-1}$.
\end{theorem}
\begin{proof}
	The proof follows from the considerations preceding the statement and from the direct computations.
\end{proof}

\begin{remark}
	Note that the inverse to $L_{i, d_1}: \mathfrak{g}(A_{d_1},\tau^{d_1}) \to \mathfrak{g}(A_{d_2},\tau^{d_2})$ is given by the following equations
	\[ L_{i, d_1}^{-1}(h_{i,d_2}) = -h_{i,d_1}, \; L_{i, d_1}^{-1}(h_{j,d_2}) = h_{i,d_1} + h_{j,d_1}, \; L_{i, d_1}^{-1}(h_{l,d_2}) = h_{l,d_1}, \]
	\[ L_{i, d_1}^{-1}(e_{i,d_2}) = f_{i,d_1}, \]
	\[ L_{i, d_1}^{-1}(f_{i,d_2}) = (-1)^{|\alpha_{i,d_1}|} e_{i,d_1}, \]
	\[ L_{i, d_1}^{-1}(e_{j,d_2}) = (x,y) (-1)^{|\alpha_{i,d_1}|(|\alpha_{i,d_1}|+|\alpha_{j,d_1}|)+[b=-1]+[b(\alpha_{i,d_1},\alpha_{j,d_1})=1]} [e_{j,d_1},e_{i,d_1}], \]
	\[ L_{i, d_1}^{-1}(f_{j,d_2}) = (-1)^{1+|\alpha_{i,d_1}||\alpha_{j,d_1}|+[b=1]+[b(\alpha_{i,d_1},\alpha_{j,d_1})=-1]} [f_{i,d_1},f_{j,d_1}], \]
	\[ L_{i, d_1}^{-1}(e_{l,d_2}) = e_{l,d_1}, \; L_{i, d_1}^{-1}(f_{l,d_2}) = f_{l,d_1}. \]
\end{remark}

\begin{corollary}
	Fix any $d_1$ and $d_2$ in $D$. Then $\mathfrak{g}(A_{d_1},\tau^{d_1})$ and $\mathfrak{g}(A_{d_2},\tau^{d_2})$ are isomorphic in $\underline{\text{sLieAlg}}$.
\end{corollary}
\begin{proof}
	The result follows from Theorem \ref{th:suplieclasisom} and the fact that Weyl groupoid $\mathcal{W}(\mathcal{R})$ is simply connected.
\end{proof}

In assumptions made in Theorem \ref{th:suplieclasisom} we prove that isomorphisms $L_{i,d}$ ($i \in I$, $d \in D$) satisfy equations
\begin{lemma}
	\label{lm:relclL}
	\begin{equation}
		L_{i, d_2}^{-} L_{i,d_1} = id_{a_{d_1}}, \; L_{i,d_1} L_{i, d_2}^{-} = id_{a_{d_2}},
	\end{equation}
	where $x =\sigma_{\alpha_{i,d_1}}^{d_1}(\alpha_{i,d_1})$;
	\begin{equation}
		L_{j,d_2} L_{i,d_1} = L_{i,d_2} L_{j,d_1},
	\end{equation}
	iff $|i-j| \ge 2$ and $|\alpha_{i,d_1}| = |\alpha_{j,d_1}| = 0$;
	\begin{equation}
		L_{j,d_4} L_{i,d_3} L_{j,d_2} L_{i,d_1} = L_{i,d_4} L_{j,d_3} L_{i,d_2} L_{j,d_1},
	\end{equation}
	iff $|i-j| \ge 2$;
	\begin{equation}
		L_{i,d_3} L_{j,d_2} L_{i,d_1} = L_{j,d_3} L_{i,d_2} L_{j,d_1},
	\end{equation}
	iff $|i-j| = 1$ and $|\alpha_{i,d_1}| = |\alpha_{j,d_1}|$;
	\begin{equation}
		L_{j,d_6} L_{i,d_5} L_{j,d_4} L_{i,d_3} L_{j,d_2} L_{i,d_1} = L_{i,d_6} L_{j,d_5} L_{i,d_4} L_{j,d_3} L_{i,d_2} L_{j,d_1},
	\end{equation}
	iff $|i-j| = 1$.
\end{lemma}
\begin{proof}
	The result follows from the straightforward computations.
\end{proof}

\begin{theorem}
	The braid goupoid of type A operates via isomorphisms on Lie superalgebras $\mathfrak{g}(A_d, \tau^{d})$ ($d \in D$).
\end{theorem}
\begin{proof}
	The result follows from Lemma \ref{lm:relclL} and Theorem \ref{th:bga}.
\end{proof}

We can endow Lie superalgebras $\mathfrak{g}(A_{d},\tau^{d})$ with the structure of a Lie superbialgebra. Recall that $A_{d}=((\alpha_{i}^{d}, \alpha_{j}^{d}))_{i,j \in I}$. Define the linear function $\delta: \mathfrak{g}(A_{d},\tau^{d}) \to \mathfrak{g}(A_{d},\tau^{d}) \otimes \mathfrak{g}(A_{d},\tau^{d})$ on the generators using formulas \eqref{eq:LSBstr}, and extend it to all elements of $\mathfrak{g}(A_{d},\tau^{d})$ using equation \eqref{eq:bisuli2} and by linearity. Then $\mathfrak{g}(A_{d},\tau^{d})$ becomes the Lie superbialgebra.

\begin{theorem}
	\label{th:unisupbialgebL}
	Let $\mathcal{W}(\mathcal{R})$ the Weyl groupoid of $sl(2|1)$. Then the Dynkin diagrams in $\mathcal{W}(\mathcal{R})$ are isomorphic as graphs if and only if corresponding Lie superbialgebras are isomorphic.
\end{theorem}
\begin{proof}
	We prove that if Dynkin diagrams are isomorphic then corresponding Lie superbialgebras are isomorphic (see Fig. \ref{pict:DD}). Consider the unique morphisms
	\[W_{d_1,d_2} \in \text{Hom}_{  \underline{\text{sBiLieAlg}} }(\mathfrak{g}(A_{d_1},\tau^{d_1}), \mathfrak{g}(A_{d_2},\tau^{d_2}))\]
	such that
	\[ W_{d_1,d_2}(h_{i,d_1}) = h_{i,d_2}, \; W_{d_1,d_2}(e_{i, d_1}) = e_{i, d_2}, \]
	\[ W_{d_1,d_2}(f_{i, d_1}) = f_{i, d_2}, \]
	for all $i \in I$ and $(d_1, d_2) \in \{ (1,2), (3,4), (5,6) \}$.
	It is easy to see that $W_{d_1,d_2}$ are isomorphisms.
	Also there exist the unique isomorphism
	\[ W_{d_1, d_2}(h_{i, d_1}) = h_{j, d_2}, \; W_{d_1, d_2}(e_{i, d_1}) = e_{j, d_2}, \]
	\[ W_{d_1, d_2}(f_{i, d_1}) = f_{j, d_2},  \]
	where $i \ne j$ ($i, j \in I$) for $(d_1, d_2) \in \{ (1,5) \}$. So the necessity follows.
	
	Now it remains to show that there is no isomorphism
	\[W_{1,3} \in \text{Hom}_{\underline{\text{sBiLieAlg}} }(\mathfrak{g}(A_{1},\tau^{1}), \mathfrak{g}(A_{3},\tau^{3}))\]
	as corresponding Dynkin diagrams are not isomorphic. Indeed, let $\phi$ be an isomorphism. As an $\phi$ is an even morphism then $\phi(e_{1,1}) \in \mathfrak{g}(A_{3},\tau^{3})_{\bar{0}}$ and we have
	\[ \phi(e_{1,1}) = \gamma_{1} [e_{1, d_2}, e_{2, d_2}] +  \gamma_{2} [f_{2, d_2}, f_{1, d_2}] + \gamma_{3} h,  \]
	where $\gamma_{1}, \gamma_{2}, \gamma_{3} \in \Bbbk$ and $h$ is an element of Cartan subalgebra. Then
	\[ \delta_{2} \circ \phi (e_{1,1}) = \]	
	\[ \gamma_{1} \frac{1}{2} ( 1 + (-1)^{|e_{1, d_2}||e_{2, d_2}|}) \cdot  \]
	\[ ( e_{2,d_2} \otimes e_{1,d_2} - e_{1,d_2} \otimes e_{2,d_2} + (h_{1,d_2} + h_{2,d_2} ) \otimes [e_{1,d_2}, e_{2,d_2}] - [e_{1,d_2}, e_{2,d_2}] \otimes (h_{1,d_2} + h_{2,d_2} ) ) + \]
	\[ \gamma_{2} \frac{1}{2} ( 1 + (-1)^{|f_{1, d_2}||f_{2, d_2}|}) \cdot  \]
	\[ ( f_{1,d_2} \otimes f_{2,d_2} - f_{2,d_2} \otimes f_{1,d_2} + (h_{1,d_2} + h_{2,d_2} ) \otimes [f_{2,d_2}, f_{1,d_2}] - [f_{2,d_2}, f_{1,d_2}] \otimes (h_{1,d_2} + h_{2,d_2} ) ) = 0; \]
	\[ \phi \otimes \phi \circ \delta_{1} (e_{1,1}) = \]
	\[ \frac{1}{2} ( \phi (h_{1,1}) \otimes \phi(e_{1,1}) - \phi(e_{1,1}) \otimes \phi(h_{1,1})) \ne 0, \]
	where $\delta_1$ ($\delta_2$) is the cocommutator for $\mathfrak{g}(A_{1},\tau^{1})$ ($\mathfrak{g}(A_{2},\tau^{2})$).
	Thus we get the contradiction.
	
	The result follows.
\end{proof}

Define the subcategory $\mathcal{SL}$ in the supercategory $\underline{\text{sLieAlg}}$ as the image of the functor $F: \mathcal{W}(\mathcal{R}) \to \underline{\text{sLieAlg}}$ defined above. Recall that objects of $\mathcal{SL}$ \eqref{eq:funcsuplieobj} are also Lie superbialgebras. Thus it follows from Proposition \ref{pr:cobraindisomLiebi} that morphisms in $\mathcal{SL}$ \eqref{eq:functsupliealgedefmor} induce structures of Lie superbialgebras on objects of $\mathcal{SL}$. Consequently, $\mathcal{SL}$ can be viewed as the subcategory in the supercategory $\underline{\text{sBiLieAlg}}$.

\section{Weyl groupoid of quantum superalgebra $sl(2|1)$ at roots of unity}

\label{sec:Wgqsru}

\subsection{Quantized universal enveloping superalgebras}

\label{subs:ques}

Here we recall the notion of quantized universal enveloping superalgebras (for more detail see \cite{Y91}, \cite{Y94}, \cite{G04}, \cite{G07}).

Let $K = \Bbbk[[h]]$, where $h$ is an indeterminate and view $K$ as a superspace concentrated in degree $\bar{0}$. Let $M$ be a module over $K$. Consider the inverse system of $K$-modules
\[ p_{n} : M_{n} = M / h^{n} M \to M_{n-1} = M / h^{n-1} M. \]
Let $\hat{M} = \varprojlim M_{n}$ be the inverse limit. Then $\hat{M}$ has the natural inverse limit topology (called the $h$-adic topology).

Let $V$ be a $\Bbbk$-superspace. Let $V[[h]]$ to be the set of formal power series. The superspace $V[[h]]$ is naturally a $K$-module and has a norm given by
\[ || v_{n} h^n + v_{n+1} h^{n+1} + ... || = 2^{-n}, \]
where $v_{n} \ne 0$ and $v_{i} \in V$ for $i \ge n$. The topology defined by this norm is complete and coincides with the $h$-adic topology. We say that a $K$-module $M$ is topologically free if it is isomorphic to $V[[h]]$ for some $\Bbbk$-module $V$.

Let $M$ and $N$ be topologically free $K$-modules. We define the topological tensor product of $M$ and $N$ to be $\widehat{M \otimes_{K} N}$ which we denote by $M \otimes N$. It follows that $M \otimes N$ is topologically free and that
\[ V[[h]] \otimes W[[h]] = (V \otimes W)[[h]] \]
for $\Bbbk$-module $V$ and $W$.

We say a (Hopf) superalgebra defined over $K$ is topologically free if it is topologically free as a $K$-module and the tensor product is the above topological tensor product.

A quantized universal enveloping (QUE) superalgebra $A$ is a topologically free Hopf superalgebra over $\Bbbk[[h]]$ such that $A / hA$ is isomorphic as a Hopf superalgebra to universal enveloping superalgebra $U(\mathfrak{g})$ for some Lie superalgebra $\mathfrak{g}$. We use the following result proved in the non-super case in \cite{D87} and in the super case in \cite{A93}.

\begin{proposition}
	\label{pr:QUEliebi}
	Let $A$ be a QUE superalgebra: $A / h A \cong U(\mathfrak{g})$. Then the Lie superalgebra $\mathfrak{g}$ has a natural structure of a Lie superbialgebra defined by
	\begin{equation}
		\label{eq:squebialgsupcon}
		\delta(x) = h^{-1} ( \Delta(\tilde{x}) - \Delta^{op}(\tilde{x}) ) \mod h,
	\end{equation}
	where $x \in \mathfrak{g}$, $\tilde{x} \in A$ is a preimage of $x$, $\Delta$ is a comultiplication in $A$ and $\Delta^{op} := \tau_{U(\mathfrak{g}), U(\mathfrak{g})} \circ \Delta$ (for the definition of $\tau_{U(\mathfrak{g}), U(\mathfrak{g})}$ see \eqref{eq:taudef}).
\end{proposition}

\begin{definition}
	\normalfont
	Let $A$ be a QUE superalgebra and let $(\mathfrak{g}, [\cdot, \cdot], \delta)$ be the Lie superbialgebra defined in Proposition \ref{pr:QUEliebi}. We say that $A$ is a quantization of the Lie superbialgebra $\mathfrak{g}$.
\end{definition}

Let $t$ be an indeterminate. Set
\[ \begin{bmatrix}m+n\\n\end{bmatrix}_{t} = \prod_{i = 0}^{n-1} \frac{t^{m+n-i} - t^{-m-n+i}}{t^{i+1} - t^{-i-1}} \in \Bbbk[t], \]
where $m, n \in \mathbb{N}$. Denote by
\begin{equation}
	\label{eq:expdef}
	e^{h t} = \sum_{n \ge 0} \frac{t^n h^n}{n!} \in \Bbbk[[h]].
\end{equation}
Put $q = e^{h / 2}$ and recall notations introduced in Section \ref{subs:dsl}. We need the following result, see \cite{KT91} and \cite{Y94}.

\begin{theorem}
	Let $(\mathfrak{g}, A, \tau)$ be a Lie superalgebra of type $A$, where Cartan matrix $A$ is symmetrizable, i. e. there are nonzero rational numbers $g_{i}$ for $i \in I$ such that $g_{i} a_{ij} = g_{j} a_{ji}$. There exists an explicit QUE Hopf superalgebra $U_{h}^{DJ}(\mathfrak{g}, A, \tau)$. The Hopf superalgebra $U_{h}^{DJ}(\mathfrak{g}, A, \tau)$ is defined as the $\Bbbk[[h]]$-superalgebra generated by the elements $h_i$, $e_{i}$ and $f_{i}$, where $i \in I$ (all generators are even except $e_{i}$ and $f_{i}$ for $i \in \tau$ which are odd), and the relations:
	\[ [h_{i}, h_{j}] = 0, \; [h_{i}, e_{j}] = a_{ij} e_{j}, \; [h_{i}, f_{j}] = - a_{ij} f_{j}, \]
	\[ [e_{i}, f_{j}] = \delta_{i,j} \frac{q^{g_{i} h_{i}} - q^{- g_{i} h_{i}} }{q^{g_{i}} - q^{-g_{i}} }, \]
	and the quantum Serre-type relations
	\[ e_{i}^{2} = f_{i}^{2} = 0 \; \text{for} \; i \in I \; \text{such that} \; a_{ii} = 0,  \]
	\[ [e_{i}, e_{j}] = [f_{i}, f_{j}] = 0 \; \text{for} \; i,j \in I \; \text{such that} \; a_{ij} = 0 \; \text{and} \; i \ne j, \]
	\[ \sum_{v = 0}^{1 + |a_{ij}|} (-1)^{v} \begin{bmatrix}1+|a_{ij}|\\v\end{bmatrix}_{q^{g_{i}}} e_{i}^{1 + |a_{ij}| - v} e_{j} e_{i}^{v} = \sum_{v = 0}^{1 + |a_{ij}|} (-1)^{v} \begin{bmatrix}1+|a_{ij}|\\v\end{bmatrix}_{q^{g_{i}}} f_{i}^{1 + |a_{ij}| - v} f_{j} f_{i}^{v} = 0 \]
	for $i \ne j$, $i \notin \tau$ and $i,j \in I$,
	\[ [ [ [e_{m-1}, e_{m}]_{q^{(\epsilon_{m}, \epsilon_{m})}} , e_{m+1} ]_{q^{(\epsilon_{m+1}, \epsilon_{m+1})}} , e_{m} ] = 0 \; \text{and} \]
	\[ [ [ [f_{m-1}, f_{m}]_{q^{(\epsilon_{m}, \epsilon_{m})}} , f_{m+1} ]_{q^{(\epsilon_{m+1}, \epsilon_{m+1})}} , f_{m} ] = 0,  \]
	if $m-1, m, m + 1 \in I$ and $a_{mm} = 0$, $\alpha_{m} = \epsilon_{m} - \epsilon_{m+1} \in \tau$. $[\cdot,\cdot]_{v}$ is the bilinear form defined by $[x,y]_{v} = xy - (-1)^{|x||y|} v y x$ on homogeneous $x, y$ and $v \in \Bbbk[[h]]$.
	The comultiplication, counit and antipode are given by
	\[ \Delta(h_{i}) = h_{i} \otimes 1 + 1 \otimes h_{i}, \; \Delta(e_{i}) = e_{i} \otimes 1 + q^{g_{i} h_{i}} \otimes e_{i}, \; \Delta(f_{i}) = f_{i} \otimes q^{-g_{i} h_{i}} + 1 \otimes f_{i}; \]
	\[ \epsilon(h_i) = \epsilon(e_{i}) = \epsilon(f_{i}) = 0; \; S(h_{i}) = -h_{i}, \; S(e_{i}) = - q^{-g_{i} h_{i}} e_{i}, \; S(f_{i}) = - f_{i} q^{g_{i} h_{i}}, \]
	where $i \in I$.
\end{theorem}

\subsection{Definition of quantum superalgebra at roots of unity}

\label{subs:dqsru}

We introduce the quantum superalgebra of $sl(m|n)$ for any Dynkin diagram using notations from section\ref{sc:Wgsl} and \ref{sc:CcLs} (see \cite{BKK00}, \cite{Y99}). Let $q$ be an algebraically independent and invertible element over $\mathbb{Q}$. Consider Lie superalgebra $\mathfrak{g}(A_{d},\tau^{d})$ ($d \in D$), where $\tau^{d} = \{ \alpha_{i,d} \; | \; i \in I \}$ is the corresponding set of simple roots and $A_{d}=((\alpha_{i,d}, \alpha_{j,d}))_{i,j \in I}$. Let $\mathfrak{U}_q^d := \mathfrak{U}_q(\mathfrak{g}(A_{d},\tau^{d}))$ for any $d \in D$ be the associative superalgebra over $\mathbb{Q}(q)$ with $1$, generated by $\{ e_{i,d}, f_{i,d}, k_{i,d}, k_{i,d}^{-1} \; | \; i \in I \}$, satisfying
\begin{equation}
	\label{eq:QUSrelf}
	XY = YX \; \text{for} \; X,Y \in \{ k_{i,d}, k_{i,d}^{-1} \; | \; i \in I \},
\end{equation}
\begin{equation}
	k_{i,d} k_{i,d}^{-1} = k_{i,d}^{-1} k_{i,d} = 1, \; e_{i,d} k_{j,d} = q^{ - (\alpha_{j}^{d}, \alpha_{i}^{d})} k_{j,d} e_{i,d}, \; k_{j,d} f_{i,d} = q^{- (\alpha_{j}^{d}, \alpha_{i}^{d})} f_{i,d} k_{j,d},
\end{equation}
\begin{equation}
	[e_{i,d},f_{j,d}]_{1} = e_{i,d} f_{j,d} - (-1)^{|\alpha_{i}^{d}| |\alpha_{j}^{d}|} f_{j,d} e_{i,d} = \delta_{i,j} \frac{k_{i,d} - k_{i,d}^{-1}}{q-q^{-1}},
\end{equation}
\begin{equation}
e_{i,d}^{2} = f_{i,d}^{2} = 0, \; \text{if} \; |\alpha_{i,d}| = 1,
\end{equation}
\begin{equation}
[e_{i,d},[e_{i,d},e_{j,d}]_{q^{(\alpha_{i,d}, \alpha_{j,d})}}]_{q^{(\alpha_{i,d}, \alpha_{i,d} + \alpha_{j,d})}} = 0, \mbox{ if } |\alpha_{i,d}| = 0, \; |i-j|=1,
\end{equation}
\begin{equation}
\label{eq:QUSrels}
[f_{i,d},[f_{i,d},f_{j,d}]_{q^{(\alpha_{i,d}, \alpha_{j,d})}}]_{q^{(\alpha_{i,d}, \alpha_{i,d} + \alpha_{j,d})}} = 0, \mbox{ if } |\alpha_{i,d}| = 0, \; |i-j|=1,
\end{equation}
\begin{equation}
	[ [ [e_{i-1,d}, e_{i,d}]_{q^{(\epsilon_{i}, \epsilon_{i})}} , e_{i+1,d} ]_{q^{(\epsilon_{i+1}, \epsilon_{i+1})}} , e_{i,d} ] = 0, \mbox{ if } |\alpha_{i,d}= \epsilon_{i} - \epsilon_{i+1}| = 1,
\end{equation}
\begin{equation}
	[ [ [f_{i-1,d}, f_{i,d}]_{q^{(\epsilon_{i}, \epsilon_{i})}} , f_{i+1,d} ]_{q^{(\epsilon_{i+1}, \epsilon_{i+1})}} , f_{i,d} ] = 0, \mbox{ if } |\alpha_{i,d}= \epsilon_{i} - \epsilon_{i+1}| = 1,
\end{equation}
where $\delta_{i,j}$ is the Kronecker delta, $\alpha_{i,d}, \alpha_{j,d} \in \tau^{d}$ ($i, j \in I$); $[\cdot,\cdot]_{v}$ is the bilinear form defined by
\begin{equation}
[x,y]_{v} = xy - (-1)^{|x||y|} v y x
\end{equation}
on homogeneous $x, y$ and $v \in \mathbb{Q}(q)$. The parity function is defined by $|k_{i,d}| = 0$ and $|e_{i, d}| = |f_{i,d}| = |\alpha_{i,d}|$, where $\alpha_{i,d} \in \tau^{d}$ ($i \in I$).

Also $\mathfrak{U}_q^d(\mathfrak{g})$ is a Hopf superalgebra which comultiplication $\Delta$, counit $\epsilon$ and antipode $S$ are
\begin{equation}
	\label{eq:standartHsupalgstrf}
	\Delta_{d}(k_{i,d}) = k_{i,d} \otimes k_{i,d}, \; \Delta_{d}(e_{i,d}) = e_{i,d} \otimes 1 + k_{i,d} \otimes e_{i,d}, \; \Delta_{d}(f_{i,d}) = f_{i,d} \otimes k_{i,d}^{-1} + 1 \otimes f_{i,d};
\end{equation}
\begin{equation}
	\label{eq:standartHsupalgstrs}
	\epsilon_{d}(k_{i,d}) = 1, \; \epsilon_{d}(e_{i,d}) = \epsilon_{d}(f_{i,d}) = 0; \; S_{d}(k_{i,d}^{\pm 1}) = k_{i,d}^{\mp 1}, \; S_{d}(e_{i,d}) = - k_{i,d}^{-1} e_{i,d}, \; S_{d}(f_{i,d}) = - f_{i,d} k_{i,d},
\end{equation}
where $i \in I$.

\begin{proposition}
	\label{pr:injqsaque}
	There exists the unique injective morphism of Hopf superalgebras for $d \in D$
	\[ f: \mathfrak{U}_q(\mathfrak{g}(A_{d},\tau^{d})) \to U_{h}^{DJ}(\mathfrak{g}, A_d, \tau), \]
	where $\tau = \{ \alpha \; | \; |\alpha| = 1, \; \alpha \in \tau^{d} \}$, such that for $i \in I$
	\[ f(q) = e^{\frac{h}{2}}, \; f(k_{i}) = e^{\frac{h h_{i}}{2}}, \; f(k_{i}^{-1}) = e^{-\frac{h h_{i}}{2}}, \; f(e_{i}) = e_{i}, \; f(f_{i}) = f_{i}. \]
\end{proposition}
\begin{proof}
	The result follows from the direct computations.
\end{proof}

Fix Hopf superalgebra $\mathfrak{U}_q(\mathfrak{g}(A_{d},\tau^{d}))$ for $d \in D$. It follows from Proposition \ref{pr:injqsaque} that we can consider $\mathfrak{U}_q^d$ as a supersubalgebra in $U_{h}^{DJ}(\mathfrak{g}, A_d, \tau)$. Thus we are able to apply equation \eqref{eq:squebialgsupcon} to $\mathfrak{U}_q^d$. Then it easy to see that the Lie superalgebra $\mathfrak{g}(A_{d},\tau^{d})$ has a natural structure of a Lie superbialgebra defined by equation \eqref{eq:LSBstr} and extended to all the elements of $\mathfrak{g}(A_{d},\tau^{d})$ using \eqref{eq:bisuli2}.

From now on let $q$ be a root of unity of odd order $p$. Then it is easy to see that $\mathfrak{U}_q^d$ can be defined in the same way. Now we introduce some auxiliary notations, see \cite{T19a}.

\begin{notation}
\label{nt:qan}
\normalfont

Let $(\mathfrak{U}_q^d)^{<}$, $(\mathfrak{U}_q^d)^{>}$, $(\mathfrak{U}_q^d)^{0}$ be the $\mathbb{Q}(q)$-subalgebras of $\mathfrak{U}_q^d$ generated respectively by $\{f_{i,d} \}_{i \in I}$, $\{ e_{i,d} \}_{i \in I}$, $\{ k_{i,d} \}_{i \in I}$.

Let $\tau^{d} = \{ \alpha_{i,d} \; | \; i \in I \}$ \eqref{eq:rootdecomp} be the set of simple roots associated with Dynkin diagram labeled by $d \in D$ and $\Delta^{+}_d = \{ \alpha_{j,d} + \alpha_{j+1,d} + ... + \alpha_{i,d} \}_{1 \le j \le i<n+m}$. Consider the following total order $"\le"$ on $\Delta^{+}_d$:
\begin{equation}
	\label{eq:orderRootquant}
	\alpha_{j,d} + \alpha_{j+1,d} + ... + \alpha_{i,d} \le \alpha_{j^{'},d} + \alpha_{j^{'}+1,d} + ... + \alpha_{i^{'},d} \mbox{ iff } j < j^{'} \mbox{ or } j = j^{'}, i \le i^{'}.
\end{equation}
For every $\beta \in \Delta^{+}_d$, we choose a decomposition $\beta = \alpha_{i_1,d}+...+\alpha_{i_p,d}$ such that $[...[e_{\alpha_{i_1,d}}, e_{\alpha_{i_2,d}},...,e_{\alpha_{i_p,d}}]$ is a non-zero root vector $e_{\beta,d}$ of $\mathfrak{g}(A_d, \tau^d)$ (here, $e_{\alpha_{i,d},d} = e_{i,d}$ denotes the standart Chevalley generator of $\mathfrak{g}(A_d, \tau^d)$). Then, we define the $PBW$ basis element $e_{\beta,d} \in (\mathfrak{U}_q^d)^{>}$ and $f_{\beta,d} \in (\mathfrak{U}_q^d)^{<}$ via
\[ e_{\beta,d} := [...[ [ e_{i_1,d}, e_{i_2,d} ]_{q^{(\alpha_{i_1,d}, \alpha_{i_2,d})}}, e_{i_3,d} ]_{q^{( \alpha_{i_1,d} + \alpha_{i_2,d}, \alpha_{i_3,d})}} , ... , e_{i_p,d} ]_{q^{(\alpha_{i_1,d} + \alpha_{i_2,d} + ... + \alpha_{i_{p-1},d}, \alpha_{i_{p},d})}}, \]
\[ f_{\beta,d} := [...[ [ f_{i_p,d}, f_{i_{p-1},d} ]_{q^{-(\alpha_{i_p,d}, \alpha_{i_{p-1},d})}}, f_{i_{p-2},d} ]_{q^{-( \alpha_{i_p,d} + \alpha_{i_{p-1},d}, \alpha_{i_{p-2},d})}} , ... , f_{i_1,d} ]_{q^{-(\alpha_{i_{p},d} + \alpha_{i_{p-1},d} + ... + \alpha_{i_{2},d}, \alpha_{i_{1},d})}}. \]
In particular, $e_{\alpha_{i,d},d} = e_{i,d}$ and $f_{\alpha_{i,d},d} = f_{i,d}$. Let $H$ denote the set of all functions $h: \tau^{d} \to \{0,1,...,p-1\}$ such that $h(\alpha) \le 1$ if $|\alpha|=1$ ($\alpha \in \Delta^{+}_d$). The monomials
\begin{equation}
	e_{h,d} :=  \overrightarrow{\prod_{\beta \in \Delta^{+}_d}} e_{\beta,d}^{h(\beta)}, \; f_{h,d} := \overleftarrow{\prod_{\beta \in \Delta^{+}_d}} f_{\beta,d}^{h(\beta)}, \; \forall h \in H
\end{equation}
will be called the ordered PBW monomials of $(\mathfrak{U}_q^d)^{>}$ and $(\mathfrak{U}_q^d)^{<}$. Here, the arrows $\rightarrow$ an $\leftarrow$ over the product signs refer to the total order \eqref{eq:orderRootquant} and its opposite, respectively.

Let $H_{0}$ denote the set of all functions $g: I \to \{0,1,...,p-1\}$. The monomials (note the order of the products is irrelevant, due to \eqref{eq:QUSrelf})
\begin{equation}
	k_{g,d} := \prod_{i \in I} k_{i,d}^{g(i)}, \; \forall g \in H_{0}.
\end{equation}
will be called the PBW monomials of $(\mathfrak{U}_q^d)^{0}$.

Thus define a total order $"\le"$ on the elements $\{ e_{\alpha,d}, f_{\alpha,d}, k_{i,d}^{\pm 1} \; | \; \alpha \in \Delta^{+}_{d}, i \in I \} \subset \mathfrak{U}_q^d$: set $k_{i,d} \le k_{j,d}$, if $i \le j$ ($i,j \in I$); $e_{\alpha,d} \le e_{\beta,d}$ and $f_{\alpha,d} \le f_{\beta,d}$, if $\alpha \le \beta$ ($\alpha, \beta \in \Delta^{+}_d$); $f_{\alpha,d} < k_{i,d} < e_{\beta,d}$ ($\alpha, \beta \in \Delta^{+}_{d}$, $i \in I$).

For all $n \in \mathbb{Z}$ set $[n]:=\frac{q^n-q^{-n}}{q-q^{-1}}$. Denote by
\[ exp_{q}(x) := \sum_{n=0}^{\infty} \frac{x^{n}}{(n)_{q}!}, \]
where $x$ is an indeterminate and for all $k \in \mathbb{N}$ we set $(k)_{q} := \frac{q^{k}-1}{q-1}$ and $(0)_{q}!:=1$, $(n)_{q}!:=(1)_{q} (2)_{q} ... (n)_{q},$ if $n \in \mathbb{Z}_{+}$.
\end{notation}

Now we define the quantum superalgebras $U_q^d := U_q(\mathfrak{g}(A_{d},\tau^{d}))$ of $sl(m|n)$ at roots of unity for the Dynkin diagrams labeled by $d \in D$, see also \cite{MS19}, Proposition 3.1.

\begin{definition}
	\label{def:Hopidealsup}
	\normalfont
	For any $d \in D$ let $U_q^d$ be the quotient of the Hopf superalgebra $\mathfrak{U}_q^d$ by the two-sided $\mathbb{Z}_{2}$-graded Hopf ideal $\mathfrak{I}$ generated by the following elements:
	\begin{equation}
		\label{eq:QUSrelt}
		e_{i,d}^p, f_{i,d}^p,
	\end{equation}
	for all $|\alpha_{i,d}| = 0$ ($\alpha_{i,d} \in \tau^d$, $i \in I$);
	\begin{equation}
		\label{eq:QUSrelfo}
		k_{i,d}^{p} - 1
	\end{equation}
	for all $i \in I$.
\end{definition}

For convenience we preserve the same notations for $U_q^d$ as for $\mathfrak{U}_q^d$ ($d \in D$). Notice that Proposition \ref{pr:QUEliebi} is not true for $U_q^d$, as when we specialize to a root of unity $q$ then the equation \eqref{eq:squebialgsupcon} doesn't hold.

\subsection{Category of Hopf superalgebras and twists}
\label{sec:chst}

We consider some categorical definitions and general results about Hopf superalgebras. Our notations here follow \cite{AM14}.

\begin{definition}
	\normalfont
	\begin{enumerate}
	\item Let $\text{Alg}$ be the strict monoidal category of algebras over the field $\mathbb{Q}(q)$. 
		
	\item Let $\text{sAlg}$ be the strict monoidal supercategory (\cite{BE17a}, Definition 1.4) of unital associative superalgebras over the field $\mathbb{Q}(q)$. A morphism $f \in \text{Hom}_{\text{sAlg}}(V,W)$ between superalgebras $(V,\mu_V,\eta_V)$ and $(W,\mu_W,\eta_W)$ is a linear map of the underlying vector spaces such that $ f \circ \mu_V = \mu_W \circ (f \otimes f)$ and $f \circ \eta_V = \eta_W$.
	
	\item Let $\text{HAlg}$ be the strict monoidal category of Hopf algebras over the field $\mathbb{Q}(q)$.
	
	\item Let $\text{sHAlg}$ be the strict monoidal supercategory of Hopf superalgebras over the field $\mathbb{Q}(q)$. A morphism $f \in \text{Hom}_{\text{sHAlg}}(V,W)$ between Hopf superalgebras $(V,\mu_{V},\eta_{V},\Delta_{V},\epsilon_{V},S_{V})$ and $(W,\mu_{W},\eta_{W},\Delta_{W},\epsilon_{W},S_{W})$ is a linear map of the underlying vector spaces such that $ f \circ \mu_V = \mu_W \circ (f \otimes f)$, $f \circ \eta_V = \eta_W$, $(f \otimes f) \circ \Delta_V = \Delta_W \circ f$, $\epsilon_W \circ f = \epsilon_V$ and $f \circ S_V = S_W \circ f$.
	\end{enumerate}
\end{definition}

Let $(H,\mu,\eta,\Delta,\epsilon,S)$ be a Hopf superalgebra in $\underline{\text{sHAlg}}$. Recall some results about twists, see \cite{D90}, \cite{KT94}, \cite{AEG01}, \cite{XZ17}.

\begin{definition}
	\normalfont
	\label{def:twisttype1}
	 A twist for $H$ is an invertible even element $\mathcal{J} \in H \otimes H$ which satisfies
	 \begin{equation}
	 	\label{eq:twistdefclf}
	 	( \Delta \otimes id_{H}) (\mathcal{J}) (\mathcal{J} \otimes 1) = (id_{H} \otimes \Delta)(\mathcal{J}) (1 \otimes \mathcal{J}),
	 \end{equation}
	 \begin{equation}
	 	\label{eq:twistdefcls}
	 	(\epsilon \otimes id_{H}) (\mathcal{J}) = (id_{H} \otimes \epsilon) (\mathcal{J}) = 1,
	 \end{equation}
	where $id_{H}$ is the identity map of $H$.
\end{definition}

\begin{proposition}
	\label{pr:twistclasHopfsupalg}
	Let $(H,\mu,\eta,\Delta,\epsilon,S)$ be a Hopf (super)algebra in $\underline{\text{HAlg}}$ ($\underline{\text{sHAlg}}$) and let $\mathcal{J}$ be a twist for $H$. Then there is a new Hopf (super)algebra $H^{\mathcal{J}} := (H,\mu,\eta,\Delta^{\mathcal{J}}, \epsilon, S^{\mathcal{J}})$ defined by the same (super)algebra and counit, and
	\[ \Delta^{\mathcal{J}}(h) := \mathcal{J}^{-1} ( \Delta(h) ) \mathcal{J}, \; S^{\mathcal{J}}(h) := U^{-1} (S(h)) U \]
	for all $h \in H$. Here $U= \mu \circ (S \otimes id_{H}) (\mathcal{J})$ and is invertible. Moreover, $U^{-1} = \mu \circ (id_{H} \otimes S) (\mathcal{J}^{-1} )$. If $H$ is a quasi-cocommutative (braided) Hopf (super)algebra with an universal $R$-matrix $R$ then $H^{\mathcal{J}}$ is also quasi-cocommutative (braided) with the universal $R$-matrix $R^{\mathcal{J}}$:
	\[ R^{\mathcal{J}} := \tau_{H, H} (\mathcal{J}^{-1}) R_{H} \mathcal{J}, \]
	where $\tau_{H, H}$ is defined by \eqref{eq:taudef}.
\end{proposition}
\begin{proof}
	The result follows from the definition and properties of a comultiplication, antipode and universal $R$-matrix.
\end{proof}

\begin{definition} \label{twist1_2}
a) The Hopf (super)algebra
\begin{equation}
	\label{eq:hstwistnotf}
	H^{\mathcal{J}} := (H,\mu,\eta,\Delta^{\mathcal{J}}, \epsilon, S^{\mathcal{J}})
\end{equation}
is called the twisted Hopf (super)algebra by the twist $\mathcal{J}$. The same notation we use for the quasi-cocommutative (braided) Hopf (super)algebra $H^{\mathcal{J}} := (H,\mu,\eta,\Delta^{\mathcal{J}}, \epsilon, S^{\mathcal{J}}, R^{\mathcal{J}})$.\\ 
b) We call $\mathcal{J}$ the twist of type $1$.\\
c) Let $\chi \in \text{Hom}_{\text{Alg}}(V,W)$ ($\text{Hom}_{\underline{\text{sAlg}}}(V,W)$ be an isomorphism. Then we call $\chi$ the twist of type $2$.
\end{definition}

\begin{proposition}
	\label{pr:nsHstr}
	Let $\chi \in \text{Hom}_{\text{Alg}}(V,W)$ ($\text{Hom}_{\underline{\text{sAlg}}}(V,W)$) be a twist of type $2$. Suppose that $V$ and $W$ are Hopf (super)algebras, namely
	\[(V,\mu_{V},\eta_{V},\Delta_{V},\epsilon_{V},S_{V}) \text{ and } (W,\mu_{W},\eta_{W},\Delta_{W},\epsilon_{W},S_{W}).\]
	Let for any $w \in W$
	\[ \Delta_{W}^{\chi}(w) := (\chi \otimes \chi) \circ \Delta_{V}(\chi^{-1}(w)), \; \epsilon_{W}^{\chi}(w) := \epsilon_{V} \circ \chi^{-1}(w), \; S^{\chi}_{W}(w) := \chi \circ S_{V}(\chi^{-1}(w)). \]
	Then $V^{\chi} := (W, \mu_{W}, \eta_{W}, \Delta_{W}^{\chi}, \epsilon_{W}^{\chi}, S^{\chi}_{W})$ is a Hopf (super)algebra isomorphic to $V$. If $V$ is a quasi-cocommutative (braided) Hopf (super)algebra with an universal $R$-matrix $R_{V}$ then $W$ is also quasi-cocommutative (braided) with the universal $R$-matrix $R_{W}^{\chi}$:
	\[ R_{W}^{\chi} = (\chi \otimes \chi)(R_{V}). \]
\end{proposition}
\begin{proof}
	The result follows from the definition of a Hopf (super)algebra morphism and direct computations.
\end{proof}

\vspace{0.3cm}

The Hopf (super)algebra
\begin{equation}
	\label{eq:hstwistnots}
	V^{\chi} := (W, \mu_{W}, \eta_{W}, \Delta_{W}^{\chi}, \epsilon_{W}^{\chi}, S^{\chi}_{W})
\end{equation}
is called the twisted Hopf (super)algebra by the isomorphism $\chi$. The same notation we use for a quasi-cocommutative (braided) Hopf (super)algebra $V^{\chi} := (W, \mu_{W}, \eta_{W}, \Delta_{W}^{\chi}, \epsilon_{W}^{\chi}, S^{\chi}_{W}, R_{W}^{\chi})$.

\vspace{0.5cm}

Let $C$ be a (super)coalgebra over a field $k$. A nonzero element $g$ in $C$ is said to be a group-like element if $\Delta(g) = g \otimes g$. The set of group-like elements in $C$ is denoted by $G(C)$. For $g, h \in G(C)$, $c \in C$ is said to be skew primitive, or more precisely, $(g, h)$-skew primitive, if $\Delta(c) = c \otimes g + h \otimes c$.
The set of all $(g, h)$-skew primitive elements is denoted by $P_{g,h}(C)$.

\subsection{Lusztig type isomorphisms}

\label{subs:LTI}

In this section we show that morphisms of supercategory $\mathcal{W}(\mathcal{R})$ can be represented by isomorphisms between the quantum superalgebras $U_q^{d}$, where $d \in D$, in supercategory $\underline{\text{sAlg}}$. Compare with the Section \ref{sc:CcLs}, see also \cite{HSTY08}, \cite{L10}.

We introduce the covariant faithful functor $F_q: \mathcal{W}(\mathcal{R}) \to \underline{\text{sAlg}}$. Fix $(A_{d}, \tau^{d}) \in \text{Obj}(\mathcal{W}(\mathcal{R}))$ for $d \in D$. The action on objects is given for all $d \in D$ by the formula
\begin{equation}
	\label{eq:Fobjqalg}
	F_q((A_d, \tau^{d}))= U_q^{d},
\end{equation}
where $A_{d}=((\alpha_{i}^{d}, \alpha_{j}^{d}))_{i,j \in I}$ and $d \in D$.

Let $\rho(a_{d_1}) = a_{d_2}$ for $d_1, d_2 \in D$. We prove in Theorem \ref{th:ltisom} that $U_q^{d_1}$ and $U_q^{d_2}$ are isomorphic in $\underline{\text{sAlg}}$. Consider a generator $\sigma_{\alpha_{i,d_1}}^{d_1} \in \text{Hom}_{\mathcal{W}(\mathcal{R})}(a_{d_1}, a_{d_2})$ \eqref{eq:WCgenel} for an appropriate $\alpha_{i,d_1} \in \tau^{d_1}$ ($i \in I$) and fix an isomorphism
\[J_{i, d_1} \in \text{Hom}_{\underline{\text{sAlg}}}(U_q^{d_1}, U_q^{d_2}).\]
Define $F_q(\sigma_{\alpha_{i,d_1}}^{d_1}) = J_{i, d_1}$ and $F_q(\sigma_{x}^{d_2}) = J_{i, d_2}$, where $x = \sigma_{\alpha_{i,d_1}}^{d_1}(\alpha_{i,d_1})$ and $J_{i, d_2} := J_{i, d_1}^{-1}$. Additionally suppose that isomorphisms $V_{i,d}$ satisfy relations \eqref{eq:invWCeqgenl} - \eqref{eq:WCeqgenl} for all $i \in I$ and $d \in D$. It is easy to see that $F_q$ is indeed the covariant faithful functor that preserves the structure of the Weyl groupoid. We give an example of the family of isomorphisms $\{ F_q(\sigma) \in \text{Hom}(\underline{\text{sAlg}}) \}_{\sigma \in \mathcal{B}}$. Call them Lusztig type isomorphisms.

\vspace{0.3cm}

Let  
\[ T_{i, d_1}, T_{i, d_1}^{-}: U_q^{d_1} \to U_q^{d_2} \]
are unique isomorphisms in $\underline{\text{sAlg}}$ satisfying the following equations  for $|i-j|=1$ and $|i-l| \ge 2$ ($j,l \in I$)
\begin{equation}
\label{eq:ltif}
T_{i, d_1}(k_{i,d_1}) = T_{i, d_1}^{-}(k_{i,d_1}) = k_{i, d_2}^{-1},
\end{equation}
\begin{equation}
T_{i, d_1}(k_{j,d_1}) = T_{i, d_1}^{-}(k_{j,d_1}) = k_{i, d_2} k_{j,d_2}, \; T_{i, d_1}(k_{l,d_1}) =  T_{i, d_1}^{-}(k_{l,d_1}) = k_{l, d_2},
\end{equation}
\begin{equation}
T_{i, d_1}(e_{i,d_1}) = (-1)^{|\alpha_{i,d_1}|} f_{i, d_2} k_{i, d_2}^{b},
\end{equation}
\begin{equation}
T_{i, d_1}(f_{i,d_1}) = k_{i,d_2}^{-b} e_{i, d_2},	
\end{equation}
\begin{equation}
T_{i,d_1}^{-}(e_{i,d_1}) = k_{i,d_2}^{-b} f_{i,d_2}
\end{equation}
\begin{equation}
T_{i,d_1}^{-}(f_{i,d_1}) = (-1)^{|x|} e_{i,d_2} k_{i,d_2}^{b},
\end{equation}
\begin{equation}
T_{i, d_1}(e_{j,d_1}) = -(x,y) [ e_{x,d_2}, e_{y,d_2} ]_{q^{z}},		
\end{equation}
\begin{equation}
T_{i, d_1}(f_{j,d_1}) = [ f_{j,d_2}, f_{i,d_2} ]_{q^{-z}},
\end{equation}
\begin{equation}
T_{i,d_1}^{-}(e_{j,d_1}) = (\alpha,\beta) (-1)^{|x|(|x|+|y|)+[b=-1]+[b(x,y)=1]} [e_{j,d_2}, e_{i, d_2}]_{q^{z}},
\end{equation}
\begin{equation}
T_{i,d_1}^{-}(f_{j,d_1}) = (-1)^{1+|x||y|+[b=1]+[b(x,y)=-1]} [f_{i,d_2}, f_{j,d_2}]_{q^{-z}},
\end{equation}
\begin{equation}
T_{i, d_1}(e_{l,d_1}) = T_{i, d_1}^{-}(e_{l,d_1}) = e_{l,d_2},
\end{equation}
\begin{equation}
\label{eq:ltis}
T_{i, d_1}(f_{l,d_1}) = T_{i, d_1}^{-}(f_{l,d_1}) = f_{l,d_2},
\end{equation}
where $x = \sigma_{\alpha_{i,d_1}}^{d_1} (\alpha_{i,d_1})$, $y = \sigma_{\alpha_{i,d_1}}^{d_1} (\alpha_{j,d_1})$, $b = [|\alpha_{i,d_1}| = 0] - [|\alpha_{i,d_1}| = 1]$ and $z = b (x,y)$.

One has $T_{d_2, d_1}^{-} = (T_{d_1,d_2})^{-1}$.

\vspace{0.3cm}

\begin{theorem}
	\label{th:ltisom}
	There exist the unique covariant faithful functor $F_q: \mathcal{W}(\mathcal{R}) \to \underline{\text{sAlg}}$ which satisfies equation \eqref{eq:Fobjqalg} and for all $\sigma_{\alpha_{i,d_1}}^{d_1} \in \mathcal{B}$ ($i \in I$)
	\begin{equation}
		\label{eq:functdefmorphq}
		F_q(\sigma_{\alpha_{i,d_1}}^{d_1}) = T_{i, d_1}, \; F_q(\sigma_{x}^{d_2}) = T_{i, d_2}^{-},
	\end{equation} 
where \[ T_{i, d_1}, T_{i, d_1}^{-}: U_q^{d_1} \to U_q^{d_2} \]
are isomorphisms in $\underline{\text{sAlg}}$ above defined  by formulas \eqref{eq:ltif} - \eqref{eq:ltis}.
\end{theorem}
\begin{proof}
	The proof follows from the considerations preceding the statement and from the direct computations.
\end{proof}

\begin{remark}
	Note that the inverse to $T_{i, d_1}: U_q^{d_1} \to U_q^{d_2}$ is given by the following equations
	\[ T_{i, d_1}^{-1}(k_{i,d_2}) = k_{i,d_1}^{-1}, \; T_{i, d_1}^{-1}(k_{j,d_2}) = k_{i,d_1} k_{j,d_1}, \; T_{i, d_1}^{-1}(k_{l,d_2}) = k_{l,d_1}, \]
	\[ T_{i, d_1}^{-1}(e_{i,d_2}) = k_{i,d_1}^{-b} f_{i,d_1}, \]
	\[ T_{i, d_1}^{-1}(f_{i,d_2}) = (-1)^{|\alpha_{i,d_1}|} e_{i,d_1} k_{i,d_1}^{b}, \]
	\[ T_{i, d_1}^{-1}(e_{j,d_2}) = (x,y) (-1)^{|\alpha_{i,d_1}|(|\alpha_{i,d_1}|+|\alpha_{j,d_1}|)+[b=-1]+[b(\alpha_{i,d_1},\alpha_{j,d_1})=1]} [e_{j,d_1}, e_{i,d_1}]_{q^{b(\alpha_{i,d_1},\alpha_{j,d_1})}}, \]
	\[ T_{i, d_1}^{-1}(f_{j,d_2}) = (-1)^{1+|\alpha_{i,d_1}||\alpha_{j,d_1}|+[b=1]+[b(\alpha_{i,d_1},\alpha_{j,d_1})=-1]} [f_{i,d_1}, f_{j,d_1}]_{q^{-b(\alpha_{i,d_1},\alpha_{j,d_1})}}, \]
	\[ T_{i, d_1}^{-1}(e_{l,d_2}) = e_{l,d_1}, \; T_{i, d_1}^{-1}(f_{l,d_2}) = f_{l,d_1}. \]
\end{remark}	

\begin{corollary}
	\label{cl:isomclass}
	Fix any $d_1$ and $d_2$ in $D$. Then $U_q^{d_1}$ and $U_q^{d_2}$ are isomorphic in $\underline{\text{sAlg}}$.
\end{corollary}
\begin{proof}
	The result follows from Theorem \ref{th:ltisom} and the fact that Weyl groupoid $\mathcal{W}(\mathcal{R})$ is simply connected.
\end{proof}

In assumptions made in Theorem \ref{th:ltisom} we prove that isomorphisms $T_{i,d}$ ($i \in I$, $d \in D$) satisfy equations
\begin{lemma}
	\label{lm:qbrel}
	\begin{equation} \label{br1}
		T_{i, d_2}^{-} T_{i,d_1} = id_{a_{d_1}}, \; T_{i,d_1} T_{i, d_2}^{-} = id_{a_{d_2}},
	\end{equation}
	where $x =\sigma_{\alpha_{i,d_1}}^{d_1}(\alpha_{i,d_1})$;
	\begin{equation}  \label{br2}
		T_{j,d_2} T_{i,d_1} = T_{i,d_2} T_{j,d_1},
	\end{equation}
	iff $|i-j| \ge 2$ and $|\alpha_{i,d_1}| = |\alpha_{j,d_1}| = 0$;
	\begin{equation} \label{br3}
		T_{j,d_4} T_{i,d_3} T_{j,d_2} L_{i,d_1} = T_{i,d_4} T_{j,d_3} T_{i,d_2} T_{j,d_1},
	\end{equation}
	iff $|i-j| \ge 2$;
	\begin{equation} \label{br4}
		T_{i,d_3} T_{j,d_2} T_{i,d_1} = T_{j,d_3} T_{i,d_2} T_{j,d_1},
	\end{equation}
	iff $|i-j| = 1$ and $|\alpha_{i,d_1}| = |\alpha_{j,d_1}|$;
	\begin{equation} \label{br5}
		T_{j,d_6} T_{i,d_5} T_{j,d_4} T_{i,d_3} T_{j,d_2} T_{i,d_1} = T_{i,d_6} T_{j,d_5} T_{i,d_4} T_{j,d_3} T_{i,d_2} T_{j,d_1},
	\end{equation}
	iff $|i-j| = 1$.
\end{lemma}
\begin{proof}
	The result follows from the straightforward computations.
\end{proof}

We will call the algebraic structure defined in the lemma \ref{lm:qbrel} by the formulas (\ref{br1} -- \ref{br5})  {\it a braided groupoid of type A}.

\begin{theorem}
	The braid groupoid of type A operates via isomorphisms on quantum superalgebras $U_q^d$ ($d \in D$).
\end{theorem}
\begin{proof}
	The result follows from Lemma \ref{lm:qbrel} and Theorem \ref{th:ltisom}.
\end{proof}

Define the subcategory $\mathcal{QS}$ in the supercategory $\underline{\text{sAlg}}$ as the image of the functor $F_q: \mathcal{W}(\mathcal{R}) \to \underline{\text{sAlg}}$ defined above. Recall that objects of $\mathcal{QS}$ \eqref{eq:Fobjqalg} are also Hopf superalgebras defined by \eqref{eq:standartHsupalgstrf} - \eqref{eq:standartHsupalgstrs}. Thus it follows from Proposition \ref{pr:nsHstr} that morphisms in $\mathcal{QS}$ \eqref{eq:functdefmorphq} are also morphisms in supercategory $\underline{\text{sHAlg}}$. Consequently, $\mathcal{QS}$ is also the subcategory in the supercategory $\underline{\text{sHAlg}}$. Recall that we defined in the analogous way the subcategory $\mathcal{SL}$ in the supercategory $\underline{\text{sBiLieAlg}}$, see Section \ref{sc:CcLs}.

\begin{proposition}
	Categories $\mathcal{QS}$ and $\mathcal{SL}$ are equivalent, where the equivalence $\mathcal{H}: \mathcal{QS} \to \mathcal{SL}$ is defined on objects by $\mathcal{H}(U_{q}^{d}) = \mathfrak{g}(A_{d},\tau^{d})$ and on morphisms by $\mathcal{H}(id_{U_{q}^{d}}) = id_{\mathfrak{g}(A_{d},\tau^{d})}$ and $\mathcal{H}(T_{i,d_1}) = L_{i,d_1}$ ($d, d_1, d_2 \in D$).
\end{proposition}
\begin{proof}
	It is easy to see that the functor $\mathcal{H}$ is full, faithful and dense. The result follows.
\end{proof}

\vspace{0.5cm}

\subsection{PBW basis of $U_q^d$}

\label{subs:pbwb}

We build for any $d \in D$ the PBW basis of $U_q^d$ in the case of $sl(2|1)$. Remind the notations and conventions introduced in \ref{nt:qan}. See also \cite{T19a}, \cite{T19b}.

\begin{theorem}
	\label{th:pbwbasis}
	The elements
	\[ \mathcal{Y} = \{ f_{h_{-}, d} \cdot k_{h_{0}, d} \cdot e_{h_{+}, d} \; | \; h_{-}, h_{+} \in H, h_{0} \in H_{0} \} \]
	form a $\mathbb{Q}(q)$-basis of the quantum superalgebra $U_q^d$, where $d \in D$.
\end{theorem}
\begin{proof}
	The statement immediately follows from the proof (\cite{MS19}, Theorem $3.1$). We need only to add extra relations \ref{eq:TLrels} and check that the result remains true. Therefore, we give only a sketch of the proof.
	
	Consider a $\mathbb{Q}(q)$ super vector space $L$ generated by $X = \{ e_{\alpha,d}, f_{\alpha,d}, k_{i,d}^{\pm 1} \; | \; \alpha \in \Delta^{+}_{d}, i \in I \}$. Introduce a pair $(T(L),i)$ where $T(L)$ is the tensor superalgebra of the vector superspace $L$ and $i$ is the canonical inclusion of $L$ in $T(L)$. We identify for convenience $X$ and $i(X)$. Rewrite equations \eqref{eq:QUSrelf} - \eqref{eq:QUSrels} and \eqref{eq:QUSrelt} - \eqref{eq:QUSrelfo} in $T(L)$ in the following way
	\begin{equation}
		\label{eq:TLrelf}
		a \otimes b - (-1)^{|a||b|} q^{\delta(a,b)} b \otimes a - [a,b]_{q^{\delta(a,b)}} = 0,
	\end{equation}
	where $a,b \in X, \; [a,b]_{q^{\delta(a,b)}} \in T(L), \; \delta: X \times X \to \{ -2,-1,0,1,2 \}$,
	\begin{equation}
		\label{eq:TLrels}
		a^{\otimes p} - c_{a} = 0,
	\end{equation}
	where $a \in X$, $|a| = 0$ and $c_{a} \in \mathbb{Q}(q) \subset T(L)$. Denote by $J$ a $\mathbb{Z}_{2}$-graded two-sided ideal in $T(L)$ generated by relations \eqref{eq:TLrelf} and \eqref{eq:TLrels}. Notice that $U_{q}^{d} \cong T(L)/J$.
	
	The index of $x_{i_1} \otimes x_{i_2} \otimes ... \otimes x_{i_n} \in T(L)$ is defined to be the number of pairs $(l,m)$ with $l < m$ but $x_{i_l} > x_{i_m}$, where $x_{i_j} \in X$, $i_j \in \Delta^{+}_d \cup I$ and $j \in \mathbb{N}$. We adopt in a natural way the definition of the index on elements of $U_q^d$. Denote by $G$ the monomials having index $0$. Notice that $G = \mathcal{Y}$ in $U_q^d$. Thus, we want to prove that $G$ forms the basis of $U_q^d$ considered as the $\mathbb{Q}(q)$-superspace.
	
	Notice that each element in $U_q^d$ is a $\mathbb{Q}(q)$-linear combination of unit and standard monomials. Indeed, it is easy to prove by induction on degree and index of elements in $U_q^d$ that this is the case.
	
	Further show that elements of $G$ are linear independent in $U_q^d$. Let $R = \mathbb{Q}(q)[z_1, ... , z_{|X|}]$ be the polynomial ring. Endow $R$ with the structure of the superalgebra by defining the parity function $|z_{i}| = |f_{\alpha,d}|$, $|z_{j+|\Delta^{+}_d|}| = |k_{j,d}|$ and $|z_{i+|I|+|\Delta^{+}_d|}| = |e_{\alpha^{'},d}|$, where $i \in \{1,...,|\Delta^{+}_d|\}$, $\alpha\in \Delta^{+}_d$ follow in descending order, $j \in I$ and $\alpha^{'} \in \Delta^{+}_d$ follow in ascending order. Now we want to construct a morphism of superspaces $U_q^d \to R$ which restriction on $G$ is a monomorphism that takes all the elements of $G$ to linear independent polynomials in $R$. Then the result follows. Thus, we proof that there is a superspace morphism $\theta:T(L) \to R$ which satisfies the following relations
	\[ \theta(1) = 1, \; \theta(f_{\alpha,d}) = z_{i}, \; \theta(k_{j,d}) = z_{j+|\Delta^{+}_d|}, \; \theta(e_{\alpha^{'},d}) = z_{i+|I|+|\Delta^{+}_d|}, \]
	where $i \in \{1,...,|\Delta^{+}_d|\}$, $\alpha\in \Delta^{+}_d$ follow in descending order, $j \in I$ and $\alpha^{'} \in \Delta^{+}_d$ follow in ascending order,
	\[ \theta(x_{i_1} \otimes x_{i_2} \otimes ... \otimes x_{i_n}) = z_{i_1} z_{i_2} ... z_{i_n}, \; \text{if } x_{i_1} \le x_{i_2} \le
	... \le x_{i_n}, \]
	\[ \theta( x_{i_1} \otimes x_{i_2} \otimes ... \otimes  x_{i_k} \otimes x_{i_{k+1}} \otimes ... \otimes x_{i_n}) -
	(-1)^{|x_{i_k}||x_{i_{k+1}}|} q^{\delta(x_{i_k},x_{i_{k+1}})} \theta( x_{i_1} \otimes x_{i_2} \otimes ... \otimes  x_{i_{k+1}}
	\otimes x_{i_k} \otimes ... \otimes x_{i_n}) = \]
	\[ = \theta( x_{i_1} \otimes x_{i_2} \otimes ... \otimes  [x_{i_k},x_{i_{k+1}}] \otimes ... \otimes x_{i_n} ) \]
	for all $x_{i_1},x_{i_2},...,x_{i_n} \in X$ and $1 \le k < n$, where $x_{i_j} \in X$, $i_j \in \Delta^{+}_d \cup I$ and $j \in \mathbb{N}$,
	\[ \theta(x^{\otimes p}) = c_{x}, \]
	where $x \in X, \; |x| = 0$ and $c_{x} \in \mathbb{Q}(q)$.
	
	Recall that $T^{0}(L) = \mathbb{Q}(q) 1$ and $T^{n}(L)= \bigotimes_{i=1}^{n} L$, where $n \in \mathbb{N}$. Denote by $T^{n,j}(L)$ a linear subspace $T^{n}(L)$ spanned by all monomials $x_{i_1} \otimes x_{i_2} \otimes ... \otimes x_{i_n}$, which have index less or equal to $j$. Thus,
	\[ T^{n,0}(L) \subset T^{n,1}(L) \subset ... \subset T^{n}(L). \]
	We define $\theta:T^{0}(L) \to R$ by $\theta(1)=1$. Suppose inductively that $\theta:T^{0}(L) \oplus T^{1}(L) ... \oplus T^{n-1}(L) \to R$ has already been defined satisfying the required conditions. We will show that $\theta$ can be extended to $\theta:T^{0}(L) \oplus T^{1}(L) ... \oplus T^{n}(L) \to R$. We define $\theta:T^{n,0}(L) \to R$ by
	\[ \theta(x_{i_1} \otimes x_{i_2} \otimes ... \otimes x_{i_n}) = z_{i_1} z_{i_2}...z_{i_n} \]
	for standard monomials of degree $n$. We suppose $\theta:T^{n,i-1} \to R$ has already been defined, thus giving a superspace morphism from $\theta:T^{0}(L) \oplus T^{1}(L) ... \oplus T^{n-1}(L) \oplus T^{n,i-1}(L) \to R$ satisfying the required 	conditions. We wish to define $\theta:T^{n,i}(L) \to R$.
	
	Assume that the monomial $x_{i_{1}}\otimes x_{i_{2}}\otimes...\otimes x_{i_{n}}$ has the index $i\ge1$ and let $x_{i_{k}}\ge x_{i_{k+1}}$. Then define
	\begin{equation}\label{eq:basisULinearMap}
	\theta(x_{i_{1}}\otimes...\otimes x_{i_{k}}\otimes x_{i_{k+1}}\otimes...\otimes
	x_{i_{n}})=\theta(x_{i_{1}}\otimes...\otimes[x_{i_{k}},x_{i_{k+1}}]\otimes...\otimes x_{i_{n}}) +
	\end{equation}
	\[ +(-1)^{|x_{i_{k}}||x_{i_{k+1}}|}q^{\delta(x_{i_{k}},x_{i_{k+1}})}\theta(x_{i_{1}}\otimes...\otimes x_{i_{k+1}}\otimes
	x_{i_{k}}\otimes...\otimes x_{i_{n}}). \]	
	This definition is correct as both terms on the right side of the equation belong to a super vector space $T^{0}(L)+T^{1}(L)+...+T^{n-1}(L)+T^{n,i-1}(L)$. We state that the definition \ref{eq:basisULinearMap} doesn't depend on the choise of the pair $(x_{i_{k}},x_{i_{k+1}})$, where $x_{i_{k}}>x_{i_{k+1}}$. Let $(x_{i_{j}},x_{i_{j+1}})$ be another pair, where $x_{i_{j}}>x_{i_{j+1}}$. There are two different possible situations: 1. $x_{i_{j}}>x_{i_{k+1}}$, 2. $x_{i_{j}}=x_{i_{k+1}}$. It is easy to see that the statement is true in both cases.
	
	Further define
	\begin{equation}\label{eq:basisULinearMaps}
	\theta(x_{i_{1}}\otimes...\otimes x_{i_{k}} \otimes x^{\otimes p} \otimes x_{i_{k+p+1}} \otimes ...\otimes x_{i_{n}}) = c_{x} \theta(x_{i_{1}}\otimes...\otimes x_{i_{k}} \otimes x_{i_{k+p+1}} \otimes ...\otimes x_{i_{n}}),
	\end{equation}
	where $p \le n$, $x \in X$, $|x| = 0$ and $c_{x} \in \mathbb{Q}(q)$. Let the monomial $x_{i_{1}}\otimes...\otimes x_{i_{k}} \otimes x^{\otimes p} \otimes x_{i_{k+p+1}} \otimes ...\otimes x_{i_{n}}$ have the index $i \ge 1$. Then it is easy to see that the order of application of equations \eqref{eq:basisULinearMap} and \ref{eq:basisULinearMaps} doesn't affect on result. Notice, in this connection, that
	\[ \theta( x^{p} \otimes y ) = \theta( y \otimes x^{p}) = c_{x} \theta(y), \]
	if $x > y$, where $x, y \in X$, $|x| = 0$ and $c_{x} \in \mathbb{Q}(q)$,
	\[ \theta( y \otimes x^{p} ) = \theta( x^{p} \otimes y) = c_{x} \theta(y), \]
	if $y > x$, where $x, y \in X$, $|x| = 0$ and $c_{x} \in \mathbb{Q}(q)$.
	
	Thus we have defined a map $\theta:T^{n,i}(L) \to R$. A linear extension of this map gives us $\theta:\sum_{j=0}^{n-1} T^{j}(L) \oplus T^{n,i}(L) \to R$, which satisfies the required conditions. Since $T^{n}= T^{n,r}$ for sufficiently large $r$, we can consider a map $\theta: \sum_{j=0}^{n} T^j(L) \to R$. Since $ T(L)=T^{0} \oplus \sum_{\substack{i \in \mathbb{N}}} T^i(L)$, we get a map $\theta: T(L) \to R$, which satisfies the required conditions. It is easy to see that $\theta:T(L)\to R$ annihilates $J$. Thus, $\theta$ induces the required superspace morphism $\bar{\theta}:T(L)/J\to R$, that is $\bar{\theta}:U_{q}^{d} \to R$.
	
\end{proof}

\subsection{Hopf superalgebra structure and universal $R$-matrix}
	\label{subs:hssurm}
	We describe how the standard Hopf superalgebra structures associated with each Dynkin diagram are related.
		
	Isomorphisms described in Theorem \ref{th:ltisom} induce Hopf superalgebra structures being twists of type $2$, see Section \ref{sec:chst}. We want to understand how the new Hopf superalgebra structure is related to the standard one defined by equations \eqref{eq:standartHsupalgstrf} - \eqref{eq:standartHsupalgstrs}.
	
	Let $\rho(a_{d_1}) = a_{d_2}$ for $d_1, d_2 \in D$. Consider a generator $\sigma_{\alpha_{i,d_1}}^{d_1} \in \text{Hom}_{\mathcal{W}(\mathcal{R})}(a_{d_1}, a_{d_2})$ \eqref{eq:WCgenel} for an appropriate $\alpha_{i,d_1} \in \tau^{d_1}$ ($i \in I$) and fix the isomorphism $F_q(\sigma_{\alpha_{i,d_1}}^{d_1}) = T_{i, d_1}$ \eqref{th:ltisom}. We put
	\[ J_{\alpha_{i,d_1}} = \sum_{i=0}^{p-1} (-1)^{i} q^{-(x,x) i(i-1)/2} \frac{(q-q^{-1})^{i}}{(i)_{q^{-(x,x)}}!} T_{i,d_1}^{i}(f_{i,d_1}) \otimes T_{i,d_1}^{i}(e_{i,d_1}) = \]
	\[ = exp_{q^{-(x,x)}} ( (q-q^{-1}) T_{i,d_1}(f_{i,d_1}) \otimes T_{i,d_1}(e_{i,d_1}) )^{-1}, \]
	if $|\alpha_{i,d_1}| = 0$;	
	\[ J_{\alpha_{i,d_1}} = 1 \otimes 1 + (q - q^{-1}) f_{i, d_{2}} \otimes e_{i, d_{2}} = exp_{q^{-(x,x)}} ( (q - q^{-1}) f_{i, d_{2}} \otimes e_{i, d_{2}} ), \]
	if $|\alpha_{i,d_1}| = 1$, where $x = \sigma_{\alpha_{i,d_1}}^{d_1} (\alpha_{i,d_1})$.
	
	\begin{lemma}
		\label{lm:twist2}
		$J_{\alpha_{i,d_1}}$ is a twist of type 1 for $(U_q^{d_2})^{T_{i,d_1}}$. Moreover, the twist of type 1 for $(U_q^{d_1})^{T_{i,d_2}^{-}}$ is given by $J_{\alpha_{i,d_2}}^{-} = (T_{i,d_2}^{-} \otimes T_{i,d_2}^{-}) (J_{\alpha_{i,d_1}}^{-1})$.
	\end{lemma}
	\begin{proof}
		First note that $J_{\alpha_{i,d_1}}$ is an even element.
		
		Let $|\alpha_{i,d_1}| = 0$. Then, the inverse is given by		
		\[ J_{\alpha_{i,d_1}}^{-1} = exp_{q^{-(x,x)}} ( (q-q^{-1}) T_{i,d_1}(f_{i,d_1}) \otimes T_{i,d_1}(e_{i,d_1}) ). \]
		Also in this case
		\[ J_{\alpha_{i,d_2}}^{-} = (T_{i,d_2}^{-} \otimes T_{i,d_2}^{-}) (J_{\alpha_{i,d_1}}^{-1}) = exp_{q^{-(x,x)}} ( (q-q^{-1}) f_{i,d_1} \otimes e_{i,d_1} ) \]
		and
		\[ (J_{\alpha_{i,d_2}}^{-})^{-1} = (T_{i,d_2}^{-} \otimes T_{i,d_2}^{-}) (J_{\alpha_{i,d_1}}) = exp_{q^{-(x,x)}} ( (q-q^{-1}) f_{i,d_1} \otimes e_{i,d_1} )^{-1}. \]		
		
		Let $|\alpha_{i,d_1}| = 1$. Then, the inverse is given by
		\[ J_{\alpha_{i,d_1}}^{-1} = exp_{q^{-(x,x)}} ( (-1) (q - q^{-1}) f_{i, d_{2}} \otimes e_{i, d_{2}} ). \]
		Thus we have
		\[ J_{\alpha_{i,d_2}}^{-} = (T_{i,d_2}^{-} \otimes T_{i,d_2}^{-}) (J_{\alpha_{i,d_1}}^{-1}) = exp_{q^{-(x,x)}} ( (q-q^{-1}) T_{i,d_2}^{-}(f_{i,d_2}) \otimes T_{i,d_2}^{-}(e_{i,d_2})  )^{-1} \]
		and
		\[ (J_{\alpha_{i,d_2}}^{-})^{-1} = exp_{q^{-(x,x)}} ( (q-q^{-1}) T_{i,d_2}^{-}(f_{i,d_2}) \otimes T_{i,d_2}^{-}(e_{i,d_2})). \]
		
		We have to check that conditions in Definiton \ref{def:twisttype1} are satisfied. The result follows from direct computations.
		
	\end{proof}
	
	Now we want to show that for the case of a quantized superalgebra $(U_q^{d_2})^{T_{i,d_1}}$ the twisting by the two-tensor $J_{\alpha_{i,d_1}}$ coincides with the $U_q^{d_2}$. We formulate	
	\begin{lemma}
		\label{lm:Hsuptwwistneigh}
		The twisted Hopf superalgebra $((U_{q}^{d_2})^{T_{i, d_1}})^{J_{\alpha_{i,d_1}}}$ coincides with Hopf superalgebra $U_q^{d_2}$.
	\end{lemma}
	\begin{proof}		
		$J_{\alpha_{i,d_1}}$ is the twist of type 1 by Lemma \ref{lm:twist2}.
		
		Now we consider how changes the coproduct and antipode of $(U_q^{d_2})^{T_{i,d_1}}$ after twisting by $J_{\alpha_{i,d_1}}$. It is easy to check that
		\[ \Delta_{d_2} (h) = J_{\alpha_{i,d_1}}^{-1} \Delta_{d_2}^{T_{i, d_1}} (h) J_{\alpha_{i,d_1}} \]
		for all $h \in U_q^{d_2}$. Since an antipode is determined uniquely by a coproduct, we have
		\[ S_{d_2}(h) = U^{-1} S_{d_2}^{T_{i, d_1}} (h) U, \]
		for all $h \in U_q^{d_2}$, where $U= \mu_{U_{q}^{d_2}} \circ (S_{d_2}^{T_{i, d_1}} \otimes id_{U_{q}^{d_2}}) (J_{\alpha_{i,d_1}})$.
		
		The result follows by Proposition \ref{pr:twistclasHopfsupalg}.
	\end{proof}
	
	Let $a_{d_1} = (A_{d_1},\tau^{d_1})$ and $ a_{d_n} = (A_{d_n},\tau^{d_n})$ be arbitrary objects in $\mathcal{W}(\mathcal{R})$ for $d_1, d_n \in D$ ($n \in \mathbb{N}$). It follows from the definition of $\mathcal{W}(\mathcal{R})$ and equations \eqref{eq:WCgenel} - \eqref{eq:WCeqgenl} that there is a morphism $\theta \in \text{Hom}_{\mathcal{W}(\mathcal{R})}(a_{d_1}, a_{d_n})$. Let $\theta = \sigma_{\alpha_{i_{n-1},d_{n-1}}}^{d_{n-1}} ... \sigma_{\alpha_{i_2,d_2}}^{d_{2}} \sigma_{\alpha_{i_1,d_1}}^{d_{1}}$, where $\sigma_{\alpha_{i_k},d_{k}}^{d_{k}} \in \text{Hom}_{\mathcal{W}(\mathcal{R})}(a_{d_{k}}, a_{d_{k+1}})$ ($i_k \in I$, $d_{k} \in D$, $\alpha_{i_k} \in \tau^{d_{k}}$ and $k, n \in \mathbb{N}$). It follows from Corollary \ref{cl:isomclass} that the functor $F_q: \mathcal{W}(\mathcal{R}) \to \underline{\text{sAlg}}$ induces a Lusztig type isomorphism $T_{d_{1}, d_{n}}: U_{q}^{d_1} \to U_{q}^{d_n} $ in $\underline{\text{sAlg}}$ such that $F_q(\theta) = T_{d_{1}, d_{n}}$ and $T_{d_{1}, d_{n}} = T_{i_{n-1}, d_{n-1}} ... T_{i_2, d_{2}} T_{i_1, d_{1}}$. 
	
	Thus we can consider Hopf superalgebra $(U_{q}^{d_1})^{P_{\theta,d_1}}$, which is defined by formula
	
	\begin{equation} \label{theta}
	 (U_{q}^{d_1})^{P_{\theta,d_1}} = (((((((U_{q}^{d_1})^{T_{i_{1},d_1}})^{J_{\alpha_{i_1,d_1}}})^{T_{i_{2},d_2}})^{J_{\alpha_{i_2,d_2}}})^{...})^{T_{i_{n-1},d_{n-1}}})^{J_{\alpha_{i_{n-1},d_{n-1}}}}. 
	\end{equation}
		
	We can summarize our considerations in the
	\begin{theorem}
		\label{th:linkHosalgstrd}
		Hopf superalgebra $ (U_{q}^{d_1})^{P_{\theta,d_1}}$	coincides with the Hopf superalgebra $U_q^{d_{n}}$.
	\end{theorem}
	\begin{proof}
		The result follows from Lemma \ref{lm:Hsuptwwistneigh} by induction on $n$.
	\end{proof}
	
	We use notations introduced in Section \ref{sec:chst}. It follows from the same arguments as in the \cite{CY12} that $G(U_{q}^{d})$ is generated by the group-like elements $k_{i}$ ($i \in I$) in the Hopf superalgebra $U_{q}^{d}$. To get all Hopf superalgebra isomorphisms of $U_{q}^{d}$ for $d \in D$, we need to describe all skew primitive elements $P_{1,g}(U_{q}^{d})$ and $P_{g,1}(U_{q}^{d})$ ($g \in G(U_{q}^{d})$).
	
	\begin{lemma}
		\label{spel}
		We have
		\begin{equation}
			\label{eqspel1}
			P_{1,g}(U_{q}^{d})=
			\begin{cases}
				\mathbb{Q}(q)(1-g) \oplus \mathbb{Q}(q) e_{i} \oplus \mathbb{Q}(q) g f_{i},& g=k_{i} \; (1 \le i < m+n), \\
				\mathbb{Q}(q)(1-g),& \mbox{otherwise};
			\end{cases}
		\end{equation}
		\begin{equation}
			\label{eqspel2}
			P_{g,1}(U_{q}^{d})=
			\begin{cases}
				\mathbb{Q}(q)(1-g) \oplus \mathbb{Q}(q) f_{i} \oplus \mathbb{Q}(q) e_{i} g,& g=k_{i}^{-1} \; (1 \le i < m+n), \\
				\mathbb{Q}(q)(1-g),& \mbox{otherwise}.
			\end{cases}
		\end{equation}
	\end{lemma}
	\begin{proof}
		The proof is similar to that in \cite{CY12}, \cite{HW10}.
	\end{proof}
	
	Fix $d_1 \in D$ and consider the quantum superalgebra $U_{q}^{d_1}$. Now we want to describe quantum superalgebras $U_{q}^{d_2}$ ($d_2 \in D$) that are isomorphic to $U_{q}^{d_1}$ as Hopf superalgebras. We can describe Dynkin diagram associated with $U_{q}^{d_1}$ ($U_{q}^{d_2}$) by the set of simple roots
	\[ \tau^{d_1} = \{ \alpha_{i,d_1} \; | \; i \in I \} \; (\tau^{d_2} = \{ \alpha_{i,d_2} \; | \; i \in I \}). \]
	
	Let $a = (a_{1}, ... , a_{m+n-1}) \in (\mathbb{Q}(q)^{*})^{m+n-1}$. Suppose that $(\alpha_{i,d_1},\alpha_{j,d_1}) = (\alpha_{i,d_2},\alpha_{j,d_2})$ for all $i,j \in I$. Define $\phi_{a}: U_{q}^{d_1} \to U_{q}^{d_2}$ by
	\begin{equation}
		\label{eq:isomgraph1}
		\phi_{a}(k_{i,d_1}) = k_{i,d_2}, \; \phi_{a}(e_{i,d_1}) = a_{i} e_{i,d_2}, \; \phi_{a}(f_{i,d_1}) = a_{i}^{-1} f_{i,d_2}.
	\end{equation}
	
	Suppose that $(\alpha_{i,d_1},\alpha_{j,d_1}) = (\alpha_{m+n-i,d_2},\alpha_{m+n-j,d_2})$ for all $i,j \in I$. We define $\phi^{'}: U_{q}^{d_1} \to U_{q}^{d_2}$ by
	\begin{equation}
		\label{eq:isomgraph3}
		\phi^{'}(k_{i,d_1}) = k_{m+n-i,d_2}, \; \phi^{'}(e_{i,d_1}) = e_{m+n-i,d_2}, \; \phi^{'}(f_{i,d_1}) = f_{m+n-i,d_2}.
	\end{equation}
	
	Suppose that $(\alpha_{i,d_1},\alpha_{j,d_1}) = -(\alpha_{i,d_2},\alpha_{j,d_2})$ for all $i,j \in I$. We define $\phi^{''}: U_{q}^{d_1} \to U_{q}^{d_2}$ by
	\begin{equation}
		\phi^{''}(k_{i,d_1}) = k_{i,d_2}, \; \phi^{''}(e_{i,d_1}) = k_{i,d_2} f_{i,d_2}, \; \phi^{''}(f_{i,d_1}) = e_{i,d_2} k_{i,d_2}^{-1}.
	\end{equation}
	
	Suppose that $(\alpha_{i,d_1},\alpha_{j,d_1}) = -(\alpha_{m+n-i,d_2},\alpha_{m+n-j,d_2})$ for all $i,j \in I$. We define $\phi^{'''}: U_{q}^{d_1} \to U_{q}^{d_2}$ by
	\begin{equation}
		\label{eq:isomgraph2}
		\phi^{'''}(k_{i,d_1}) = k_{m+n-i,d_2}, \; \phi^{'''}(e_{i,d_1}) = k_{m+n-i,d_2} f_{m+n-i,d_2}, \; \phi^{'''}(f_{i,d_1}) = e_{m+n-i,d_2} k_{m+n-i,d_2}^{-1}.
	\end{equation}
		
	\begin{theorem}
		\label{th:sufcondaut}
		Let $\phi_{a}$, $\phi^{'}$, $\phi^{''}$ and $\phi^{'''}$ be as above. Then these maps define Hopf superalgebra isomorphisms.
	\end{theorem}
	\begin{proof}
		It follows immediately from the definition that $\phi_{a}$, $\phi^{'}$, $\phi^{''}$ and $\phi^{'''}$ are $\mathbb{Z}_{2}$-grading and even maps. The result follows from the direct computations.
	\end{proof}
	
	Compare the following result with Theorem 4 in \cite{CY12}.
	
	\begin{theorem}
		\label{th:classHopf}
		Let $U_{q}^{d_1}$ and $U_{q}^{d_2}$ be Hopf superalgebras associated with Dynkin diagrams $d_1$ and $d_2 \in D$ resprectively. Then there exists an isomorphism $\phi \in \text{Hom}_{\underline{\text{sHAlg}}}(U_{q}^{d_{1}},U_{q}^{d_{2}})$ if and only if $d_1$ and $d_2$ are isomorphic as graphs and
		\begin{enumerate}
			\item $\phi = \phi_{a}$, if $(\alpha_{i,d_1},\alpha_{j,d_1}) = (\alpha_{i,d_2},\alpha_{j,d_2})$ for all $i,j \in I$; \label{cd:hsisom1}
			\item $\phi = \phi^{'} \circ \phi_{a}$, if $(\alpha_{i,d_1},\alpha_{j,d_1}) = (\alpha_{m+n-i,d_2},\alpha_{m+n-j,d_2})$ for all $i,j \in I$; \label{cd:hsisom2}
			\item $\phi = \phi^{''} \circ \phi_{a}$, if $(\alpha_{i,d_1},\alpha_{j,d_1}) = -(\alpha_{i,d_2},\alpha_{j,d_2})$ for all $i,j \in I$; \label{cd:hsisom3}
			\item $\phi = \phi^{'''} \circ \phi_{a}$, if $(\alpha_{i,d_1},\alpha_{j,d_1}) = -(\alpha_{m+n-i,d_2},\alpha_{m+n-j,d_2})$ for all $i,j \in I$. \label{cd:hsisom4}
		\end{enumerate}
	\end{theorem}
	\begin{proof}
		For any $i \in I$, we have
		\[ \Delta_{d_2}(\phi(e_{i,d_1})) = (\phi \otimes \phi) \Delta_{d_1}(e_{i,d_1}) = \phi( e_{i,d_1} ) \otimes 1 + \phi( k_{i,d_1} ) \otimes \phi( e_{i,d_1} ). \]
		It follows from the definition of the quantum superalgebra in Section \ref{subs:dqsru} that there are no other invertible elements in $U_{q}^{d}$ except for that in $G(U_{q}^{d})$ for $d \in D$. Thus $\phi(k_{i,d_1}) \in G(U_{q}^{d_2})$. Moreover, as $\phi$ is an isomorphism, we have from \eqref{eqspel1} that there exists some $\bar{i} = \rho(i) \in I$ such that $\phi(k_{i,d_1}) = k_{\bar{i},d_2}$ for $\rho: I \to I$. It is clear that $\rho$ is bijective.
		By Lemma \ref{spel} we have
		\[ \phi(e_{i,d_1}) = a_{i} e_{\bar{i}, d_2} + b_{i} k_{\bar{i}, d_2} f_{\bar{i},d_2} + c_{i} (1-k_{\bar{i}, d_2}), \]
		for suitable $a_{i}, b_{i}, c_{i} \in \mathbb{Q}(q)$. Note that $e_{\bar{i}, d_2}$, $k_{\bar{i}, d_2} f_{\bar{i},d_2}$ and $1-k_{\bar{i}, d_2}$ are linear independent.
		
		Recall that $q^2 \ne 1$. Consider the equation
		\begin{equation}
			\label{eqThf1}
			\phi( k_{i,d_1} e_{j,d_1} k_{i,d_1}^{-1} ) = q^{(\alpha_{i,d_1},\alpha_{j,d_1})} \phi(e_{j,d_1}).
		\end{equation}
		We have
		\begin{equation}
			\label{eqThf2}
			0 = a_{j} ( q^{(\alpha_{\bar{i},d_2},\alpha_{\bar{j},d_2})} - q^{(\alpha_{i,d_1},\alpha_{j,d_1})} ) = b_{j} ( q^{-(\alpha_{\bar{i},d_2},\alpha_{\bar{j},d_2})} - q^{(\alpha_{i,d_1},\alpha_{j,d_1})} ) = c_{j} (1 - q^{(\alpha_{i,d_1},\alpha_{j,d_1})}).
		\end{equation}
		Let $i \ne j$. Note that $(\alpha_{\bar{i},d_2},\alpha_{\bar{j},d_2}) = 0, 1 \mbox{ or } -1$. There is at least one odd vertex in Dynkin diagram. Let $|e_{j,d_1}| = 1$ and $|i-j|=1$. Then $|(\alpha_{i,d_1},\alpha_{j,d_1})| = |(\alpha_{\bar{i},d_2},\alpha_{\bar{j},d_2})| = 1$. Thus $c_{j} = 0$ and $a_{j} \ne 0$, $b_{j} = 0$ or $a_{j} = 0$, $b_{j} \ne 0$. Thus if $|e_{j,d_1}| = 1$
		\[ \phi(e_{j,d_1}) = a_{j} e_{\bar{j},d_2} \mbox{ or } \phi(e_{j,d_1}) = b_{j} k_{\bar{j},d_2} f_{\bar{j},d_2}. \]
		Moreover,
		\begin{equation}
			\label{eqThf3}
			\phi(e_{j,d_1}^2) = a_{j}^2 e_{\bar{j},d_2}^2 = 0 \mbox{ or } \phi(e_{j,d_1}) = q^{(\alpha_{\bar{j},d_2}, \alpha_{\bar{j},d_2})} b_{j}^2 k_{\bar{j},d_2}^2 f_{\bar{j},d_2}^2 = 0.
		\end{equation}
		It follows that $\phi(e_{j,d_1})$ is an odd generator. Denote by $I = I_{\bar{0},d_1} \cup I_{\bar{1},d_1} = I_{\bar{0},d_2} \cup I_{\bar{1},d_2}$ the sets of even and odd indices of generators in $U_{q}^{d_{1}}$ and $U_{q}^{d_{2}}$ respectively. Then we get from \eqref{eqThf3} that $|I_{\bar{1},d_2}| \ge |I_{\bar{1},d_1}|$. As $\phi$ is the isomorphism using the same argument, we have $|I_{\bar{1},d_1}| \ge |I_{\bar{1},d_2}|$. Thus $|I_{\bar{1},d_1}| = |I_{\bar{1},d_2}| \Rightarrow |I_{\bar{0},d_1}| = |I_{\bar{0},d_2}|$.
		
		Consider the equation \eqref{eqThf1} and let $i = j$ and $|e_{j,d_1}| = 0$. It follows that $\phi(e_{j,d_1})$ is an even generator and $|(\alpha_{i,d_1},\alpha_{j,d_1})| = |(\alpha_{\bar{i},d_2},\alpha_{\bar{j},d_2})| = 2$. Then we get from the equation \eqref{eqThf2} that if $|e_{j,d_1}| = 0$
		\[ \phi(e_{j,d_1}) = a_{j} e_{\bar{j},d_2} \mbox{ or } \phi(e_{j,d_1}) = b_{j} k_{\bar{j},d_2} f_{\bar{j},d_2}. \]
		
		Case 1. Fix $j \in I$ such that $|e_{j,d_1}| = 1$ and let $\phi(e_{j,d_1}) = a_{j} e_{\bar{j},d_2}$. We claim that $\forall i \in I$, we have $\phi(e_{i,d_1}) = a_{i} e_{\bar{i},d_2}$ and $\rho(i) = i$ or $\rho(i) = m+n-i$.
		
		Consider the equations
		\[ \phi( k_{i,d_1} e_{j,d_1} k_{i,d_1}^{-1} ) = q^{(\alpha_{i,d_1},\alpha_{j,d_1})} \phi(e_{j,d_1}), \; k_{\bar{i},d_2} a_{j} e_{\bar{j},d_2} k_{\bar{i},d_2}^{-1} = q^{(\alpha_{\bar{i},d_2},\alpha_{\bar{j},d_2})} a_{j} e_{\bar{j},d_2}. \]
		Thus $\forall i \in I$
		\[ q^{(\alpha_{\bar{i},d_2},\alpha_{\bar{j},d_2})} = q^{(\alpha_{i,d_1},\alpha_{j,d_1})}, \; (\alpha_{\bar{i},d_2},\alpha_{\bar{j},d_2}) = (\alpha_{i,d_1},\alpha_{j,d_1}). \]
		Moreover, if $|i-j|=1$, we have  $(\alpha_{i,d_1},\alpha_{j,d_1}) \ne 0$ and  $\bar{i} = \bar{j} -1$ or $\bar{i} = \bar{j} +1$.
		
		It follows from
		\[ \phi( k_{j,d_1} e_{i,d_1} k_{j,d_1}^{-1} ) = q^{(\alpha_{i,d_1},\alpha_{j,d_1})} \phi(e_{i,d_1}) \]
		that for $|i-j|=1$
		\[ 0 = a_{i} ( q^{(\alpha_{\bar{i},d_2},\alpha_{\bar{j},d_2})} - q^{(\alpha_{i,d_1},\alpha_{j,d_1})} ) = b_{i} ( q^{-(\alpha_{\bar{i},d_2},\alpha_{\bar{j},d_2})} - q^{(\alpha_{i,d_1},\alpha_{j,d_1})} ). \]
		Hence, $\phi(e_{i,d_1}) = a_{i} e_{\bar{i},d_2}$.
		
		Now we suppose that $\bar{i} = \bar{j} - 1$. By the same argument $\overline{i-1} = \bar{i} - 1$ or $\overline{i-1} = \bar{i} + 1 = \bar{j}$. But $\rho$ is the bijection and this implies that $\overline{i-1} = \bar{i} - 1$.	Thus we use induction on $r$ to prove that $\phi(e_{r,d_1}) = a_{i} e_{\bar{r},d_2}$ in the case $1 \le r \le j$; the case $j + 1 \le r \le m+n-1$ is similar. Using the same arguments as before it is easy to see that the result follows. Moreover, $\forall i \in I$ as
		\[ \rho(1) < \rho(2) < ... < \rho(j) < ... < \rho(m+n-1) \]
		we have $\rho(i) = i$.
		
		Now we suppose that $\bar{i} = \bar{j} + 1$.  By the same argument, we have $\overline{i-1} = \bar{i} + 1$. Thus we use induction on $r$ to prove that $\phi(e_{r,d_1}) = a_{i} e_{\bar{r},d_2}$ in the case $1 \le r \le j$; the case $j + 1 \le r \le m+n-1$ is similar. Using the same arguments as before it is easy to see that the result follows. Moreover, $\forall i \in I$ as
		\[ \rho(1) > \rho(2) > ... > \rho(j) > ... > \rho(m+n-1) \]
		we have $\rho(i) = m+n-i$.
		
		Case 2. Fix $j \in I$ such that $|e_{j,d_1}| = 1$ and let $\phi(e_{j,d_1}) = b_{j} k_{\bar{j},d_2} f_{\bar{j},d_2}$. We claim that $\forall i \in I$, we have $\phi(e_{i,d_1}) = b_{i} k_{\bar{i},d_2} f_{\bar{i},d_2}$ and $\rho(i) = i$ or $\rho(i) = m+n-i$.
		
		Consider the equation
		\[ \phi( k_{i,d_1} e_{j,d_1} k_{i,d_1}^{-1} ) = q^{(\alpha_{i,d_1},\alpha_{j,d_1})} \phi(e_{j,d_1}), \; k_{\bar{i},d_2} b_{j} k_{\bar{j},d_2} f_{\bar{j},d_2} k_{\bar{i},d_2}^{-1} = q^{-(\alpha_{\bar{i},d_2},\alpha_{\bar{j},d_2})} b_{j} k_{\bar{j},d_2} f_{\bar{j},d_2}. \]
		Thus $\forall i \in I$
		\[ q^{(\alpha_{i,d_1},\alpha_{j,d_1})} = q^{-(\alpha_{\bar{i},d_2},\alpha_{\bar{j},d_2})}, \; (\alpha_{i,d_1},\alpha_{j,d_1}) = -(\alpha_{\bar{i},d_2},\alpha_{\bar{j},d_2}). \]
		Moreover, if $|i-j|=1$, we have  $(\alpha_{i,d_1},\alpha_{j,d_1}) \ne 0$ and  $\bar{i} = \bar{j} -1$ or $\bar{i} = \bar{j} +1$.
		
		It follows from
		\[ \phi( k_{j,d_1} e_{i,d_1} k_{j,d_1}^{-1} ) = q^{(\alpha_{i,d_1},\alpha_{j,d_1})} \phi(e_{i,d_1}) \]
		that for $|i-j|=1$
		\[ 0 = a_{i} ( q^{(\alpha_{\bar{i},d_2},\alpha_{\bar{j},d_2})} - q^{(\alpha_{i,d_1},\alpha_{j,d_1})} ) = b_{i} ( q^{-(\alpha_{\bar{i},d_2},\alpha_{\bar{j},d_2})} - q^{(\alpha_{i,d_1},\alpha_{j,d_1})} ). \]
		Hence, $\phi(e_{i,d_1}) = b_{i} k_{\bar{i},d_2} f_{\bar{i},d_2}$.
		
		Now we suppose that $\bar{i} = \bar{j} - 1$. By the same argument $\overline{i-1} = \bar{i} - 1$ or $\overline{i-1} = \bar{i} + 1 = \bar{j}$. But $\rho$ is the bijection and this implies that $\overline{i-1} = \bar{i} - 1$.	Thus we use induction on $r$ to prove that $\phi(e_{r,d_1}) = b_{r} k_{\bar{r},d_2} f_{\bar{r},d_2}$ in the case $1 \le r \le j$; the case $j + 1 \le r \le m+n-1$ is similar. Using the same arguments as before it is easy to see that the result follows. Moreover, $\forall i \in I$ as
		\[ \rho(1) < \rho(2) < ... < \rho(j) < ... < \rho(m+n-1) \]
		we have $\rho(i) = i$.
		
		Now we suppose that $\bar{i} = \bar{j} + 1$.  By the same argument, we have $\overline{i-1} = \bar{i} + 1$. Thus we use induction on $r$ to prove that $\phi(e_{r,d_1}) = b_{r} k_{\bar{r},d_2} f_{\bar{r},d_2}$ in the case $1 \le r \le j$; the case $j + 1 \le r \le m+n-1$ is similar. Using the same arguments as before it is easy to see that the result follows. Moreover, $\forall i \in I$ as
		\[ \rho(1) > \rho(2) > ... > \rho(j) > ... > \rho(m+n-1) \]
		we have $\rho(i) = m+n-i$.
		
		Consider the equation
		\[\phi( k_{i,d_1} f_{j,d_1} k_{i,d_1}^{-1} ) = q^{-(\alpha_{i,d_1},\alpha_{j,d_1})} \phi(f_{j,d_1}).  \]
		We use \eqref{eqspel2} to obtain that $\phi(f_{j,d_1}) = a_{j}^{'} f_{j,d_2} $ or $\phi(f_{j,d_1}) = a_{j}^{'} f_{m+n-j,d_2} $ in Case 1 and $\phi(f_{j,d_1}) = b_{j}^{'} e_{j, d_2} k_{j,d_2}^{-1}$ or $\phi(f_{j,d_1}) = b_{j}^{'} e_{m+n-j, d_2} k_{m+n-j,d_2}^{-1}$ in Case 2.
		
		Finally, consider the equation
		\[ \phi( e_{i,d_1} f_{j,d_1} - (-1)^{|\alpha_{i,d_1}| |\alpha_{j,d_1}|} f_{j,d_1} e_{i,d_1} ) =  \delta_{i,j} \phi( \frac{k_{i,d_1} - k_{i,d_1}^{-1}}{q-q^{-1}} ). \]
		We obtain that $a_{i}^{'} = a_{i}^{-1}$ and $b_{i}^{'} = b_{i}^{-1}$ for all $i \in I$.
		
		This impiles that if there exists an isomorphism $\phi \in \text{Hom}_{\underline{\text{sHAlg}}}(U_{q}^{d_{1}},U_{q}^{d_{2}})$ then Dynkin diagrams associated with $U_{q}^{d_{1}}$ and $U_{q}^{d_{2}}$ are isomorphic as graphs and
		
		\begin{enumerate}
			\item $\phi = \phi_{a}$, if $(\alpha_{i,d_1},\alpha_{j,d_1}) = (\alpha_{i,d_2},\alpha_{j,d_2})$ for all $i,j \in I$;
			\item $\phi = \phi^{'} \circ \phi_{a}$, if $(\alpha_{i,d_1},\alpha_{j,d_1}) = (\alpha_{m+n-i,d_2},\alpha_{m+n-j,d_2})$ for all $i,j \in I$;
			\item $\phi = \phi^{''} \circ \phi_{a}$, if $(\alpha_{i,d_1},\alpha_{j,d_1}) = -(\alpha_{i,d_2},\alpha_{j,d_2})$ for all $i,j \in I$;
			\item $\phi = \phi^{'''} \circ \phi_{a}$, if $(\alpha_{i,d_1},\alpha_{j,d_1}) = -(\alpha_{m+n-i,d_2},\alpha_{m+n-j,d_2})$ for all $i,j \in I$.
		\end{enumerate}
	
		The sufficience follows from Theorem \ref{th:sufcondaut}.
		
	\end{proof}

	Let $\text{Aut}_{\underline{\text{sHAlg}}}(U_{q}^{d})$ ($d \in D$) be the group of Hopf superalgebra automorphisms of $U_{q}^{d}$ with the associated Dynkin diagram described by the set of simple roots $ \tau^d = \{ \alpha_{i,d} \; | \; i \in I \}$. Consider the conditions imposed on roots described in the formulation of Theorem \ref{th:classHopf}. Also recall that $q$ is a root of unity of odd order $p$.
		
	\begin{corollary}		
		\begin{enumerate}
			\item $\text{Aut}_{\underline{\text{sHAlg}}}(U_{q}^{d}) \cong \mathbb{Z} / 2 \mathbb{Z} \times (\mathbb{Q}(q)^{*})^{m+n-1}$, if the condition \ref{cd:hsisom2} is satisfied; \label{cd:aut1}
			\item $\text{Aut}_{\underline{\text{sHAlg}}}(U_{q}^{d}) \cong \mathbb{Z} / 2p \mathbb{Z} \times (\mathbb{Q}(q)^{*})^{m+n-1}$, if the condition \ref{cd:hsisom4} is satisfied; \label{cd:aut2}
			\item $\text{Aut}_{\underline{\text{sHAlg}}}(U_{q}^{d}) \cong (\mathbb{Q}(q)^{*})^{m+n-1}$, otherwise. \label{cd:aut3}
		\end{enumerate}
	\end{corollary}
	\begin{proof}
		We use the same notations as in Theorem \ref{th:classHopf}.
		
		Consider the case \ref{cd:aut1} (\ref{cd:aut2}). Let $\mathbb{Z} / 2 \mathbb{Z} = <x>$ ($\mathbb{Z} / 2p \mathbb{Z} = <x>$). Recall that in the condition \ref{cd:hsisom2} (\ref{cd:hsisom4}) of Theorem \ref{th:classHopf} are explicitly descibed all possible automorphisms for this case. Thus we are able to construct the function $ \gamma_{1} : \text{Aut}_{\underline{\text{sHAlg}}}(U_{q}^{d}) \to \mathbb{Z} / 2 \mathbb{Z} \times (\mathbb{Q}(q)^{*})^{m+n-1} $ ($\gamma_{2} : \text{Aut}_{\underline{\text{sHAlg}}}(U_{q}^{d}) \to \mathbb{Z} / 2 p \mathbb{Z} \times (\mathbb{Q}(q)^{*})^{m+n-1}$) defined by
		\[ \gamma_{1} ( \phi_{a} ) = 1 \times a, \; \gamma_{1} ( \phi^{'} ) = x \times 1 \; ( \gamma_{2} ( \phi_{a} ) = 1 \times a, \; \gamma_{2} ( \phi^{'''} ) = x \times 1 ).   \]
		Also note that $\phi_{a}$ and $ \phi^{'}$ ($ \phi^{'''}$) commute. Applying $\phi^{'}$ ($ \phi^{'''}$) to the generators of $U_{q}^{d}$ (see \eqref{eq:isomgraph3} and \eqref{eq:isomgraph2} respectively) we see that $\phi^{'}$ has order $2$ ($ \phi^{'''}$ has order $2p$). It is easy to see that $\gamma_{1}$ ($\gamma_{2}$) induces the group isomorphism.
		
		Finally consider the case \ref{cd:aut3}. Note that functions described by the condition \ref{cd:hsisom3} of Theorem \ref{th:classHopf} cannot exist  when we consider the automorphisms of Hopf superalgebra structure on $U_{q}^{d}$. Thus there is only one possible case left, namely the condition \ref{cd:hsisom1} of Theorem \ref{th:classHopf}. Construct the function $ \gamma_{3} : \text{Aut}_{\underline{\text{sHAlg}}}(U_{q}^{d}) \to (\mathbb{Q}(q)^{*})^{m+n-1} $ defined by
		\[ \gamma_{3} ( \phi_{a} ) = a. \]
		It is clear that $\gamma_{3}$ induces the group isomorphism.
	\end{proof}
	
	We use notations introduced in Section \ref{sc:Wgsl}. Consider the set of simple roots $ \tau^d = \{ \alpha_{i,d} \; | \; i \in I \}$ that corresponds to the Dynkin diagram of $U_{q}^{d}$ ($d \in D$). Recall that $\epsilon_{i}$ ($i \in \overline{1,m}$) and $\delta_{i}$ ($i \in \overline{1,n}$) are elements of the dual space $\mathfrak{d}^{*}$. Denote by $\bar{\epsilon}_{i}$ ($i \in I(m|n)$) an element of $L \cup L^{'}$, where $L = \{ \epsilon_{i} \; | \; i \in \overline{1,m} \}$ and $L^{'} = \{ \delta_{i} \; | \; i \in \overline{1,n} \}$. Let $S_{m}$ ($S_{n}$) be the symmetric group. Define the action of $S_{m} \times S_{n}$ over $L \cup L^{'}$: if $(s_1, s_2) \in S_{m} \times S_{n}$ and $\bar{\epsilon}_{i} \in L \cup L^{'}$ ($i \in I(m|n)$) then	
	\[ (s_1, s_2) \diamond (\bar{\epsilon}_{i}) =
	\begin{cases}
		\epsilon_{s_1(i)},& \mbox{if } \; \bar{\epsilon}_{i} \in L, \\
		\delta_{s_2(i)},& \mbox{if } \; \bar{\epsilon}_{i} \in L^{'}.
	\end{cases} \]
	We can use this definition to induce the corresponding action of $S_{m} \times S_{n}$ over $\tau^d$ in the obvious way, i. e. to induce the function $\diamond : (S_{m} \times S_{n}) \times \tau^d \to \tau^d$.
	
	Additionally  consider the case $m=n$. We define the map $\boxplus : L \cup L^{'} \to L \cup L^{'}$ in the following way: if $\bar{\epsilon}_{i} \in L \cup L^{'}$ ($i \in I(m|n)$) then
	\[ \boxplus (\bar{\epsilon}_{i}) =
	\begin{cases}
		\delta_{i},& \mbox{if } \; \bar{\epsilon}_{i} \in L, \\
		\epsilon_{i},& \mbox{if } \; \bar{\epsilon}_{i} \in L^{'}.
	\end{cases} \]
	We can use this definition to induce the function $\boxplus : \tau^d \to \tau^d$ in the obvious way.
	
	It is easy to prove the following result
	
	\begin{theorem}
		\label{th:dynkthunhs}
		
	A) Fix $d \in D$. Consider the Dynkin diagram of the Hopf superalgebra $U_{q}^{d}$
	   \[ ( \bar{\epsilon}_{1} - \bar{\epsilon}_{2}, \bar{\epsilon}_{2} - \bar{\epsilon}_{3}, ..., \bar{\epsilon}_{m+n-1} - \bar{\epsilon}_{m+n} ). \]	
	   Then Dynkin diagrams of all isomorphic Hopf superlagebras have the form
		\begin{enumerate}
			\item
			\[ ( s \diamond ( \bar{\epsilon}_{1} - \bar{\epsilon}_{2} ), s \diamond (\bar{\epsilon}_{2} - \bar{\epsilon}_{3}), ..., s \diamond (\bar{\epsilon}_{m+n-1} - \bar{\epsilon}_{m+n} ) ); \]
			\label{en:dd1}
			\item
			\[ ( s \diamond ( \bar{\epsilon}_{m+n} - \bar{\epsilon}_{m+n-1} ), ..., s \diamond (\bar{\epsilon}_{3} - \bar{\epsilon}_{2}), s \diamond (\bar{\epsilon}_{2} - \bar{\epsilon}_{1} ) ); \]
			\item
			\[ ( s \diamond ( \boxplus( \bar{\epsilon}_{1} ) - \boxplus( \bar{\epsilon}_{2} ) ), s \diamond ( \boxplus(\bar{\epsilon}_{2} )- \boxplus( \bar{\epsilon}_{3}) ), ..., s \diamond ( \boxplus(\bar{\epsilon}_{m+n-1}) - \boxplus (\bar{\epsilon}_{m+n}) ) ), \]
			if $m=n$ and all simple roots in $\tau^d$ are odd;
			\item
			\[ ( s \diamond ( \boxplus( \bar{\epsilon}_{m+n} ) - \boxplus ( \bar{\epsilon}_{m+n-1} ) ), ..., s \diamond ( \boxplus ( \bar{\epsilon}_{3} ) - \boxplus ( \bar{\epsilon}_{2} ) ), s \diamond ( \boxplus ( \bar{\epsilon}_{2} ) - \boxplus ( \bar{\epsilon}_{1} ) ) ), \]
			if $m=n$;
			\label{en:dd4}
		\end{enumerate}
			for all $s \in S_{m} \times S_{n}$.\\
			
		B) Hopf superalgebra structures of $U_{q}(sl(m|n))$ are enumerated by the Dynkin diagrams of $U_{q}^{d}$ ($d \in D$) up to isomorphism in the sense of Theorem \ref{th:classHopf}.
	\end{theorem}
	\begin{proof}
		First we prove part A. We want to construct Dynkin diagrams associated with all Hopf superalgebras $U_{q}^{d^{'}}$ ($d^{'} \in D$) that are isomorphic to $U_{q}^{d}$. Therefore we consider each isomorphism described in Theorem \ref{th:classHopf} and construct Dynkin diagrams associated with isomorphic Hopf superalgebras. It is easy to see that isomorhisms described in conditions \ref{cd:hsisom1} - \ref{cd:hsisom4} in Theorem \ref{th:classHopf} give rise to Dynkin diagrams described by conditions \ref{en:dd1} - \ref{en:dd4} stated in part A respectively. The result follows.

		Now we prove part B. Fix $d \in D$. Apply Theorem \ref{th:classHopf} to get all Hopf superalgebras $U_{q}^{d^{'}}$ ($d^{'} \in D$) isomorphic to $U_{q}^{d}$. Recall Corollary \ref{cl:isomclass} which states that there exists an isomorphism of superalgebras $ \gamma_{d^{'}} : U_{q}^{d^{'}} \to U_{q}(sl(m|n))$ for all $d^{'} \in D$. By Proposition \ref{pr:nsHstr} we are able to endow $U_{q}(sl(m|n))$ with the structure of Hopf superalgebra which is isomorphic to that on $U_{q}^{d}$ using $\gamma_{d}$ as a twist of type $2$ (see Definition \ref{twist1_2}). Thereom \ref{th:classHopf} guarantees us that in this way we get all non-isomorphic structures of Hopf superalgebra on $U_{q}(sl(m|n))$ which are are induced by $U_{q}^{d^{'}}$ for all $d^{'} \in D$. The result follows.
	\end{proof}

\vspace{0.5cm}

\subsection{Unique Hopf superalgebra structures of $U_{q}(sl(2|1))$}

We describe all unique Hopf superalgebra structures of $U_{q}(sl(2|1))$ and their universal $R$-matices using Theorem \ref{th:dynkthunhs} and Fig. \ref{pict:DD}.

Condsider Lie superalgebra $sl(2|1)$. Recall that $\tau^{1} = \{ \alpha_{1,1} = \epsilon_{1} - \epsilon_{2}, \alpha_{2,1} = \epsilon_{2} - \delta_{1} \}$, $\alpha_{3,1} = \alpha_{1,1} + \alpha_{2,1}$ and $\tau^{3} = \{ \alpha_{1,3} = \epsilon_{1} - \delta_{1}, \alpha_{2,3} = \delta_{1} - \epsilon_{2} \}$, $\alpha_{3,3} = \alpha_{1,3} + \alpha_{2,3}$ (see Fig. \eqref{pict:DD}).

Notice that there are just two non-isomorphic quantum Hopf superalgebras of $U_{q}(sl(2|1))$ associated with Dynkin diagrams. Indeed, all quantum Hopf superalgebras associated with Dynkin diagrams labeled by $d \in \{ 1, 2, 5, 6 \}$ are isomorphic by Theorem \ref{th:dynkthunhs}. We get the same result for quantum Hopf superalgebras associated with Dynkin diagrams labeled by $d \in \{ 3,4 \}$. Moreover, we know from Theorem \ref{th:dynkthunhs} that quantum Hopf superalgebras $U_q^{d_1}$ and $U_q^{d_2}$ with $d_1 \in \{ 1, 2, 5, 6 \}$ and $d_2 \in \{ 3,4 \}$ are non-isomorphic (compare with Theorem \ref{th:unisupbialgebL}).

We know that the universal $R$-matrix $\bar{R}_{1}$ of $U_q(sl(2|1))$ (see \cite{MS19}, Theorem 3.4) is the even invertible element
\[ \bar{R}_{1} = \tilde{R} K, \]
where
\[ \tilde{R} = exp_{q^2}( (q-q^{-1}) [e_{1,1}, e_{2,1}]_{q^{-1}} \otimes [f_{2,1}, f_{1,1}]_{q} ) exp_{q^2} ((q-q^{-1}) e_{2,1} \otimes f_{2,1} ) \times \]
\[ \times exp_{q^2}((-1) (q-q^{-1}) e_{1,1} \otimes f_{1,1} ) exp_{q^2}( (-1) q (q-q^{-1})^{3} [e_{1,1}, e_{2,1}]_{q^{-1}} e_{2,1} \otimes [f_{2,1}, f_{1,1}]_{q} f_{2,1} ), \]
\[ K = p^{-2} \sum_{0 \le i_1, j_1, i_2, j_2 \le p-1} q^{i_1(2i_2-j_2)-j_1i_2} k_{1,1}^{i_2} k_{2,1}^{j_2} \otimes k_{1,1}^{i_1} k_{2,1}^{j_1}. \]

It follows from Theorem \ref{th:linkHosalgstrd} that $R$-matrix $ \bar{R_{3}}$ for $U_{q}^{3}$ with the standard Hopf superalgebra structure defined by equations \eqref{eq:standartHsupalgstrf} - \eqref{eq:standartHsupalgstrs} is
\[ \bar{R_{3}} = (\tau_{U_q^3, U_q^3} \circ J_{\alpha_{2,1}}^{-1}) \bar{R}^{T_{2,1}} J_{\alpha_{2,1}}, \]
where
\[ \bar{R}^{T_{2,1}} = ( T_{2,1} \otimes T_{2,1} ) (\bar{R}_{1}) = \tilde{R}^{T_{2,1}} K^{T_{2,1}}, \]
\[ \tilde{R}^{T_{2,1}} = exp_{q^2}( (q-q^{-1}) [ [e_{2,3}, e_{1,3}]_{q^{-1}}, f_{2,3} k_{2,3}^{-1} ]_{q^{-1}} \otimes [ k_{2,3} e_{2,3} , [f_{1,3}, f_{2,3}]_{q} ]_{q}  ) \times \]
\[ exp_{q^2} ( (q-q^{-1}) (-1) f_{2,3} k_{2,3}^{-1} \otimes k_{2,3} e_{2,3} ) \times \]
\[ exp_{q^2}( (q-q^{-1}) [e_{2,3}, e_{1,3}]_{q^{-1}} \otimes [f_{1,3}, f_{2,3}]_{q} ) \times \]
\[ exp_{q^2}( q (q-q^{-1})^{3} [[e_{2,3}, e_{1,3}]_{q^{-1}}, f_{2,3} k_{2,3}^{-1} ]_{q^{-1}} f_{2,3} k_{2,3}^{-1} \otimes [ k_{2,3} e_{2,3} , [f_{1,3}, f_{2,3}]_{q} ]_{q} k_{2,3} e_{2,3} ), \]
\[ K^{T_{2,1}} = p^{-2} \sum_{0 \le i_1, r_1, i_2, r_2 \le p-1} q^{i_1 r_2 + i_2 r_1} k_{1,3}^{i_2} k_{2,3}^{r_2} \otimes k_{1,3}^{i_1} k_{2,3}^{r_1}, \]
\[ J_{\alpha_{2,1}} = 1 \otimes 1 + (q - q^{-1}) f_{2,3} \otimes e_{2,3}, \]
\[ J_{\alpha_{2,1}}^{-1} = 1 \otimes 1 - (q - q^{-1}) f_{2,3} \otimes e_{2,3}. \]

It follows from the direct computations that
\[ \bar{R_{3}} = (\tau_{U_q^3, U_q^3} \circ J_{\alpha_{2,1}}^{-1}) \tilde{R}^{T_{2,1}} J_{\alpha_{2,1}} K^{T_{2,1}}. \]

\vspace{0.5cm}

\section*{Acknowledgements}

The authors are grateful to Egor Dotsenko for a discussion of the results of this article, as well as for clarifying the relation between the quantum Weyl groupoid and the Casimir connection.

\section*{Funding}

This work is funded by Russian Science Foundation, scientific grant 21-11-00283. 

\end{document}